\documentclass[12pt,leqno]{amsart}
\usepackage{amsmath, amssymb, amscd, amsfonts, stmaryrd, turnstile, mathrsfs, eucal,color}

\definecolor{dullmagenta}{rgb}{0.4,0,0.4}

\definecolor{darkblue}{rgb}{0,0,0.4}

\usepackage[colorlinks=true, pdfstartview=FitV, linkcolor=red, citecolor=blue, urlcolor=darkblue]
{hyperref}

\newtheorem{theorem}{Theorem}[section]
\newtheorem{lemma}[theorem]{Lemma}
\newtheorem{proposition}[theorem]{Proposition}
\newtheorem{corollary}[theorem]{Corollary}

\theoremstyle{definition}
\newtheorem{definition}[theorem]{Definition}
\newtheorem{example}[theorem]{Example}
\newtheorem{remark}[theorem]{Remark}

\begin{document}

\title[distributive  extension ]{ Distributive FCP extensions}

\author[G. Picavet and M. Picavet]{Gabriel Picavet and Martine Picavet-L'Hermitte}
\address{Math\'ematiques \\
8 Rue du Forez, 63670 - Le Cendre\\
 France}
\email{picavet.mathu (at) orange.fr}

\begin{abstract} We are dealing with extensions of commutative rings $R\subseteq S$ whose chains of the poset $[R,S]$ of their subextensions are finite ({\em i.e.} $R\subseteq S$ has the FCP property) and such that $[R,S]$ is a distributive lattice, that we call distributive FCP extensions. Note that the lattice $[R,S]$ of a distributive FCP extension is finite. This paper is the continuation of our earlier papers where we studied catenarian and Boolean extensions. Actually, for an FCP extension, the following implications hold: Boolean $\Rightarrow$ distributive $\Rightarrow$ catenarian. A comprehensive  characterization of distributive FCP extensions actually remains a challenge, essentially because the same problem for field extensions is not completely solved. Nevertheless, we are able to exhibit a lot of positive results for some classes of extensions. A main result is that an FCP extension $R\subseteq S$ is distributive if and only if $R\subseteq\overline R$ is distributive, where $\overline R$ is the integral closure of $R$ in $S$. A special attention is paid to distributive field extensions. 
\end{abstract} 


\keywords {FIP, FCP extension, minimal extension, integral extension, u-closure, subintegral closure, support of a module, length of an extension, distributive lattice, Boolean lattice, algebraic field extension, }

\maketitle

\section{Introduction and Notation}

We consider the category of commutative and unital rings, whose epimorphisms will be involved. If $R\subseteq S$ is a (ring) extension, we denote by $[R,S]$ the set of all $R$-subalgebras of $S$ and call it the lattice of the extension. We set $]R,S[: =[R,S]\setminus \{R,S\}$ (with a similar definition for $[R,S[$ or $]R,S]$). 
 A {\it subextension} $T\subseteq U$ of $R\subseteq S$ is a (ring) extension $T\subseteq U$ such that  $T,U\in[R,S]$.
Any unexplained terminology can be found in the next sections and as a general rule, an extension $R\subseteq S$ is said to have some property of lattices if $[R,S]$ has this property. We use the lattice definitions and properties described in \cite{NO}. 
 
Actually for an extension $R\subseteq S$, the poset $([R,S],\subseteq)$ is a {\it complete} lattice, where the supremum of any non-void subset is the compositum of its elements, which we call {\it product} from now on and denote by $\Pi$ when necessary, and the infimum of any non-void subset is the intersection of its elements. We are aiming to study the distributivity of the lattice $[R,S]$ in case it is both Artinian and Noetherian. A nice result states that a lattice is distributive if and only it contains neither a diamond nor a pentagon \cite[Theorem 1, page 59]{G}. This characterization is difficult to use because of its negative formulation. We will give in the context of lattices of some ring extensions more concrete characterizations, adapted to the lattices we consider. 

This paper is  the continuation of our earlier papers \cite{Pic 10}, \cite{Pic 11} and \cite{Pic 12}, on lattice properties of ring extensions. In \cite{Pic 10}, we already gave some properties of distributive ring extensions. Our aim is to give some characterizations and to give more properties of distributive extensions for FCP extensions described below. The reader is warned  the subject is very complex and although  we have substantial results, it is clear that our study is not complete. It is enough to consider the case of finite field extensions.
    
Actually for an FCP extension, the following implications hold: Boo-lean $\Rightarrow$ distributive $\Rightarrow$ catenarian, and we have already examined the first and third properties. Our process may seem weird but all these properties are interwoven.

A {\it chain} of $[R,S]$ is a set of elements of $[R,S]$ that are pairwise comparable with respect to inclusion and $R\subseteq S$ is called {\it chained} (also called $\lambda$-extension by some authors) if $[R,S]$ is a chain. We also say that the extension $R\subseteq S$ has FCP (or is an FCP extension) if each chain in $[R,S]$ is finite, or equivalently, its lattice is both Noetherian and Artinian. Dobbs and the authors characterized FCP  extensions \cite{DPP2}. 
  
Our main tool are the minimal (ring) extensions, a concept that was introduced by Ferrand-Olivier \cite{FO}. Recall that an extension $R\subset S$ is called {\it minimal} if $[R, S]=\{R,S\}$. An extension $R\subseteq S$ is called {\it simple} if $S=R[t]$ for some $t\in S$. A minimal extension is simple. The key connection between the above ideas is that if $R\subseteq S$ has FCP, then any maximal (necessarily finite) chain $\mathcal C$ of $R$-subalgebras of $S$, $R=R_0\subset R_1\subset\cdots\subset R_{n-1}\subset R_n=S$, with {\it length} $\ell(\mathcal C):=n <\infty$, results from juxtaposing $n$ minimal extensions $R_i\subset R_{i+1},\ 0\leq i\leq n-1$. An FCP extension is finitely generated, and (module) finite if integral. For any extension $R\subseteq S$, the {\it length} $\ell[R,S]$ of $[R,S]$ is the supremum of the lengths of chains of $R$-subalgebras of $S$. Notice  that if $R\subseteq S$ has FCP, then there {\it does} exist some maximal chain of $R$-subalgebras of $S$ with length $\ell[R,S]$ \cite[Theorem 4.11]
{DPP3}.

Lattices that are catenarian and have a finite length are finite. It follows that the lattice of a distributive FCP extension is such that $[R,S]$ is finite; that is, the extension has FIP (for the ``finitely many intermediate algebras property"), a property also characterized in \cite{DPP2}. Clearly, each extension that satisfies FIP must also satisfy FCP. 
 We give a characterization of the distributivity of catenarian extensions.

 \subsection{ A summary of the main results} Any undefined  material
 is explained at the end of the section or in the next sections. 

Section 2 is devoted to recalls and results on lattices, mainly in the context of FCP and FIP extensions, and to give first properties of distributive extensions. In this respect, we give rules concerning the usual algebraic constructions in commutative ring theory. We prove that an arbitrary Pr\"ufer extension is distributive (Theorem \ref{5.41}). Extensions of Nagata idealizations of modules are examined.

In Section 3, we look at the general properties of FCP distributive ring extensions related to various algebraic situations. A key result is Proposition \ref{6.1}, translated from \cite[Theorem, p.16]{LS}, which states that an FCP extension $R\subset S$ is distributive if and only if, for any $T,U\in[R,S]$, the following conditions hold: If $T\cap U\subset U$ and $T\cap U\subset T$ are minimal, then $|[T\cap U,TU]|=4$ and if $T\subset TU$ and $ U\subset TU$ are minimal, then $|[T\cap U,TU]|=4$. We consider pinched FCP extensions, showing that their distributivity is equivalent to the distributivity of each of its paths. We give an application to the Loewy series of some special extensions by means of Boolean extensions. A length $2$ extension is distributive if and only if its lattice has less than $4$ elements. We also consider $ B_2$-extensions (Boolean of length $2$), giving a characterization involving minimal extensions and other materials of ring theory recalled in Section 1.

As a consequence of the distributivity of arbitrary Pr\"ufer extensions, we get the first main result: the non trivial Theorem \ref{6.5}. It asserts that an FCP extension $R\subset S$ is distributive if and only if $R\subseteq\overline R$ is distributive, where $\overline R$ is the integral closure of $R$ in $S$. This result allows us to reduce our work to the integral FCP extensions. 

The complexity of this study needs three sections to get  general results. It is known that there are three types of minimal extensions: inert, decomposed and ramified. 

Let $R\subset S$ be an FCP extension and $\mathcal C$ be a maximal (necessarily finite) chain of $R$-subalgebras of $S$, $R=R_0\subset R_1\subset\cdots\subset R_{n-1}\subset R_n=S$. We say that $\mathcal C$ is {\it isotopic} if all minimal subextensions $R_i\subset R_{i+1}$ are of the same type for all $i\in\{0,\ldots,n-1\}$. The extension $R\subset S$ is said an {\it isotopic FCP} (IFCP) extension if all minimal subextensions of $R\subset S$ are of the same type. It turns out that the subextensions of an FCP extension, determined by its seminormalization and $t$-closure, are IFCP extensions. Their distributivity admits  simple characterizations explained in Section 4. If $R\subset S$ is a subintegral (resp. seminormal and infra-integral) extension, then Proposition \ref{6.7} (resp.; \ref{decomp}) shows that $R\subset S$ is distributive if and only if $R\subset S$ is arithmetic (resp.; locally minimal). If the extension is $t$-closed, then it is distributive if and only if its residual field extensions are distributive. This phenomenon appears each time we consider lattice properties; for example, when considering catenarity.

We then  consider the case of a local base ring $R$ of an integral FCP extension $R\subseteq S$. There are two cases determined by the finite maximal spectrum of $S$: the unbranched case where $|\mathrm{Max}(S)|=1$ and the branched case where $|\mathrm{Max}(S)| > 1$.

Locally unbranched extensions are characterized in Section 5 through the $t$-closure $T$ of the extension: the extension is distributive if and only if it is pinched at $T$, $T\subseteq S$ is distributive and $R\subseteq T$ is arithmetic (Theorem \ref{6.11}).

Branched extensions over a local ring are examined in Section 6. Note that when such an extension is distributive, then necessarily $|\mathrm{Max}(S)| = 2$. In \cite{Pic 14}, we studied different types of closures of a ring extension, in particular, the u-closure and the co-subintegral closure. These closures play a main role in the characterization of distributive integral extensions as in Theorem \ref{6.15}, which gives a necessary and sufficient condition for an integral  branched FCP extension to be distributive. The formulation of this result is too long to be developed in this section. It is illustrated by an example.
 
When an FCP extension $R\subseteq S$ is integral and distributive, we prove in Section 7 that the fiber of each prime ideal of $R$ has at most two elements (Proposition \ref{7.0}). If this last condition is satisfied and in addition the extension is seminormal and infra-integral with a locally maximal conductor, the converse holds and the extension is locally minimal decomposed. Let an extension $R\subseteq R[u]$, with $u^2-u,u^3-u^2\in R$ (an elementary $u$-integral extension), then the fibers of the extension have at most $2$ elements and $S/I\simeq R/I\times R//I$. This last statement looks like the characterization of a minimal decomposed extension. But in general the situation is much more complicated. The rest of Section 7 examines the relationship between the above mentioned properties. 

In Section 8, we consider numerical properties of a distributive integral FCP (whence FIP) extension $R\subset S$, with the calculation of $\ell[R,S]$ and $|[R,S]|$. 

The most complicated part, namely the t-closed part of an extension, can be reduced to the study of field extensions. This is the object of Section 9. We get simple results for two types of field extensions. Let $k\subset L$ be an FCP field extension. If $k\subset L$ is radicial, then $k\subset L$ is distributive if and only if $k\subset L$ is chained (Proposition \ref{6.21}). If $k\subset L$ is Galois, then $k\subset L$ is distributive if and only if $k\subset L$ is a cyclic extension (Proposition \ref{6.23}). Now a complete characterization of distributive finite field extensions is certainly a challenge as for catenarian field extensions.

The paper ends with Section 10 with an application of $\Delta$-extensions (extensions whose lattices are stable under addition)  to distributive extensions.
 
\subsection{Conventions and notation} 
A {\it local} ring is here what is called elsewhere a quasi-local ring. As usual, Spec$(R)$ and Max$(R)$ are the set of prime and maximal ideals of a ring $R$. The support of an $R$-module $E$ is $\mathrm{Supp}_R(E):=\{P\in\mathrm{Spec }(R)\mid E_P\neq 0\}$, and $\mathrm{MSupp}_R(E):=\mathrm{Supp}_R(E)\cap\mathrm{Max}(R)$. When $R\subseteq S$ is an extension, we will set $\mathrm{Supp}_R(T/R):=\mathrm{Supp}(T/R)$ and $\mathrm{Supp}_R(S/T):=\mathrm{Supp}(S/T)$ for each $T\in [R,S]$, unless otherwise specified. A ring extension $R\subseteq S$ is called an {\it $i$-extension} if the natural map $\mathrm{Spec}(S)\to\mathrm{Spec}(R)$ is injective.
 
If $R\subseteq S$ is a ring extension and $P\in\mathrm{Spec}(R)$, then $S_P$ is both the localization $S_{R\setminus P}$ as a ring and the localization at $P$ of the $R$-module $S$. We denote by $\kappa_R(P)$ the residual field $R_P/PR_P$ at $P$. If $R\subseteq S$ is a ring extension and $Q\in\mathrm{Spec}(S)$, there exists a residual field extension $\kappa_R(Q\cap R)\to\kappa_S(Q)$. An extension $R\subset S$ is called {\it locally minimal} if $R_P\subset S_P$ is minimal for each $P\in\mathrm{Supp}(S/R)$  or equivalently for each $P\in\mathrm{MSupp}(S/R)$.  
   
We denote by $(R:S)$ the conductor of $R\subseteq S$. The integral closure of $R$ in $S$ is denoted by $\overline R^S$ (or by $\overline R$ if no confusion can occur). 
  
The characteristic of an integral domain $k$ is denoted by $\mathrm{c}(k)$. In this paper, a  purely inseparable field extension is called {\it radicial}. A field extension $k\subset L$ is said to be {\it exceptional} if $k$ is the radicial closure of $k$ in $L$ and $L$ is not the separable closure of $k\subset L$ \cite{GHe} . 

Finally, $|X|$  is the cardinality of a set $X$, $\subset$ denotes proper inclusion and and for a positive integer $n$, we set $\mathbb{N}_n:=\{1,\ldots,n\}$.  

\subsection{Minimal extensions}
The following notions and results are deeply involved in the sequel. 
\begin{definition}\label{crucial 1}\cite[Definition 2.10]{Pic 7} An extension $R\subset S$ is called {\it crucial} if $|\mathrm{Supp}(S/R)|= 1$, in which case the only element $M$ of this support is called the {\it crucial (maximal) ideal} $\mathcal{C}(R,S)$ of $R\subset S$ and the extension is called $M$-crucial. 
\end{definition}

\begin{theorem}\label{crucial}\cite[Th\'eor\`eme 2.2]{FO} A  minimal extension  is crucial  and is either integral (finite) or a flat epimorphism.
\end{theorem} 

Three types of minimal integral extensions exist, characterized in the next theorem, (a consequence of  the fundamental lemma of Ferrand-Olivier), so that there are four types of minimal extensions, mutually exclusive.

\begin{theorem}\label{minimal} \cite [Theorem 2.2]{DPP2} Let $R\subset T$ be an extension and $M:=(R: T)$. Then $R\subset T$ is minimal and finite if and only if $M\in\mathrm{Max}(R)$ and one of the following three conditions holds:

\noindent (a) {\bf inert case}: $M\in\mathrm{Max}(T)$ and $R/M\to T/M$is a minimal field extension.

\noindent (b) {\bf decomposed case}: There exist $M_1,M_2\in\mathrm{Max}(T)$ such that $M= M _1\cap M_2$ and the natural maps $R/M\to T/M_1$ and $R/M\to T/M_2$ are both isomorphisms.

\noindent (c) {\bf ramified case}: There exists $M'\in\mathrm{Max}(T)$ such that ${M'}^2 \subseteq M\subset M',\  [T/M:R/M]=2$, and the natural map $R/M\to T/M'$ is an isomorphism.

In each of the above cases, $M=\mathcal{C}(R,T)$.
\end{theorem}

\begin{lemma}\label{1.9} \cite[Corollary 3.2]{DPP2} If there exists a maximal chain $R=R_0\subset\cdots\subset R_i \subset\cdots\subset R_n=S$ of extensions, such that each $R_i\subset R_{i+1}$ is minimal, then $\mathrm{Supp}(S/R)=\{ \mathcal C (R_i, R_{i+1})\cap R\mid i=0,\ldots,n-1\}$.
\end{lemma}

In particular, if $R\subset S$ is an FCP extension, then $\mathrm{Supp}(S/R)$ has finitely many elements.

\subsection{Pr\"ufer and quasi-Pr\"ufer extensions}

Recall that an extension $R\subseteq S$ is called {\it Pr\"ufer} if $R\subseteq T$ is a flat epimorphism for each $T\in[R,S]$ (or equivalently, if $R\subseteq S$ is a normal pair) \cite[Theorem 5.2, p. 47]{KZ}. It follows that a Pr\"ufer integral extension is an isomorphism. The {\it Pr\"ufer hull} of an extension $R\subseteq S$ is the greatest {\it Pr\"ufer} subextension $\widetilde R$ of $[R,S]$ \cite{Pic 3}. An extension $R\subseteq S$ is called {\it almost-Pr\"ufer} if $\widetilde R\subseteq S$ is integral, or equivalently, if $S=\widetilde R\overline R$ \cite[Theorem 4.6]{Pic 5}. An extension $R\subseteq S$ is called {\it quasi-Pr\"ufer} if $\overline R\subseteq S$ is Pr\"ufer \cite[Definition 2.1]{Pic 5}. According to \cite[Corollary 3.4]{Pic 5}, an FCP extension is quasi-Pr\"ufer.
  
In \cite{Pic 5}, we called an extension which is a minimal flat epimorphism, a {\it Pr\"ufer minimal} extension. 

\subsection{On the seminormalization, t-closure and u-closure}
  
\begin{definition}\label{1.3} An integral extension $R\subseteq S$ is called {\it infra-integral} \cite{Pic 2} (resp.; {\it subintegral} \cite{S}) if all its residual extensions $\kappa_R(P)\to \kappa_S(Q)$, (with $Q\in\mathrm {Spec}(S)$ and $P:=Q\cap R$) are isomorphisms (resp$.$; and the natural map $\mathrm {Spec}(S)\to\mathrm{Spec}(R)$ is bijective). An extension $R\subseteq S$ is called {\it t-closed} (cf. \cite{Pic 2}) if the relations $b\in S,\ r\in R,\ b^2-rb\in R,\ b^3-rb^2\in R$ imply $b\in R$. The $t$-{\it closure} ${}_S^tR$ of $R$ in $S$ is the smallest element $B\in [R,S]$  such that $B\subseteq S$ is t-closed and the greatest element  $B'\in [R,S]$ such that $R\subseteq B'$ is infra-integral. An extension $R\subseteq S$ is called {\it seminormal} (cf. \cite{S}) or {\it s-closed} \cite[Definition 1.5]{Pic 0} if the relations $b\in S,\ b^2\in R,\ b^3\in R$ imply $b\in R$. The {\it seminormalization} ${}_S^+R$ of $R$ in $S$ is the smallest element $B\in [R,S]$ such that $B\subseteq S$ is seminormal and the greatest  element $ B'\in[R,S]$ such that $R\subseteq B'$ is subintegral. 
  The {\it canonical decomposition} of an arbitrary ring extension $R\subset S$ is $R \subseteq {}_S^+R\subseteq {}_S^tR \subseteq \overline R \subseteq S$. 
  
An extension $R\subseteq S$ is called {\it u-closed or anodal} (cf.  \cite[Definition 1.5]{Pic 0}) if the relations $b\in S,\ b^2-b\in R,\ b^3-b^2\in R$ imply $b\in R$. If $R\subseteq S$ is an integral FCP extension, then \cite[Proposition 5.2]{Pic 14} says that $R\subseteq S$ is u-closed if and only if $R\subseteq S$ is an $i$-extension. The {\it u-closure} ${}_S^uR$ of $R$ in $S$ is the smallest element $B\in [R,S]$ such that $B\subseteq S$ is u-closed. If $R\subseteq S$ is an integral FCP extension, then ${}_S^uR$  is the smallest element $B\in[R,S]$ such that $B\subseteq S$ is an $i$-extension. It follows that when $R\subseteq S$ is an infra-integral FCP extension, ${}_S^uR$ is the smallest element $B\in[R,S]$ such that $B\subseteq S$ is subintegral.
    
\begin{remark}\label{1.300} Let $R\subseteq S$ be an integral FCP extension. In \cite[Theorem 5.8]{Pic 14}, we proved that ${}_S^uR{}_S^+R={}_S^tR$. This leads to the following commutative diagram:
 $$\begin{matrix}
{} & {} &              {}                &      {}       & {}_S^uR &       {}     & {}      & {}  & {} \\
{} & {} &              {}                &\nearrow &      & \searrow & {} &{}  & {} \\
R&\to&{}_S^uR\cap{}_S^+R&{} & {} & {} & {}_S^uR{}_S^+R={}_S^tR & \to & S \\
{} & {} &              {}               &\searrow &        {}      & \nearrow & {}     & {}  & {} \\
{} & {} &              {}               &      {}      &  {}_S^+R &       {}       & {}     & {}  & {}
\end{matrix}$$
We get a decomposition which is dual to the canonical decomposition. While $R\subseteq {}_S^+R$ is an IFCP extension (all its minimal subextension are ramified), it is not the case for $R\subseteq{}_S^uR$. In fact, $R\subseteq{}_S^uR\cap{}_S^+R$ and ${}_S^uR\subseteq{}_S^tR$ are subintegral, ${}_S^uR\cap{}_S^+R\subseteq{}_S^uR$ is seminormal infra-integral and ${}_S^tR\subseteq  S$ is t-closed.
\end{remark}
    
We gave the following dual definition of seminormalization in \cite{Pic 14}. If there exists a least $T\in[R,S]$ such that $T\subseteq S$ is subintegral, we say that $S_R^+:=T$ is the {\it co-subintegral closure} of $R\subseteq S$  of $S$ in $R$. 
 We proved that the u-closure and the t-closure are linked in the following way:
\begin{proposition}\label{1.301} \cite[Corollary 6.7 and Theorem 6.6]{Pic 14} Let $R\subseteq S$ be an integral FCP extension. Then, ${}_S^uR=({}_S^tR)^+_R$, and ${}_S^uR=S^+_R$ when $R\subseteq S$ is an infra-integral FCP extension.
 \end{proposition}
 
When $R\subseteq S$ is an FCP infra-integral extension, the co-subintegral closure always exists \cite[Theorem 6.6]{Pic 14}. But a proper co-subintegral closure may not exist. See, for example \cite[Remark 6.3]{Pic 14}.
    \end{definition}  
 
The next proposition describes the link between the elements of the canonical decomposition and minimal extensions.

\begin{proposition}\label{1.31} \cite[Proposition 2.10]{Pic 13} Let there be an integral extension $R\subset S$ and a maximal chain $\mathcal C$ of $R$-subextensions of $S$, defined by $R=R_0\subset\cdots\subset R_i\subset\cdots\subset R_n= S$, where each $R_i\subset R_{i+1}$ is  minimal,  for some $n\in\mathbb N$.
The following statements hold: 

\begin{enumerate}
\item $R\subset S$ is subintegral if and only if each $R_i\subset R_{i+1}$ is  ramified. 

\item $R\subset S$ is seminormal and infra-integral if and only if each  $R_i\subset R_{i+1}$ is decomposed. 

\item  $R\subset S$ is  infra-integral if and only if each  $R_i\subset R_{i+1}$ is either decomposed or ramified. 

\item $R \subset S$ is t-closed if and only if  each $R_i\subset R_{i+1}$ is inert. 
\end{enumerate}
If either (1) or (4) holds, then $R\subset S$ is an $i$-extension.

In particular, if $R$ is a local ring, then $|\mathrm{Max}(S)|<\infty$.
\end{proposition} 
     
\begin{corollary}\label{1.312} Let $R\subset S$ be an FCP extension. Then $R\subset S$ is an IFCP extension if and only if there exists a maximal chain of $R$-subalgebras of $S$, $R=R_0\subset R_1\subset\cdots\subset R_{n-1}\subset R_n=S$ which  is  isotopic. 
\end{corollary} 

\begin{proof} Use the canonical decomposition, Proposition \ref{1.31} and \cite[Proposition 1.3]{Pic 5}. 
\end{proof}

An IFCP integral extension is either subintegral, or both seminormal and infra-integral, or t-closed. A Pr\"ufer FCP extension is an IFCP extension.
 
\begin{corollary}\label{1.311} Let $R\subset S$ be an integral FCP extension. Then $R\subset S$ is an $i$-extension if and only if ${}_S^+R= {}_S^tR$. 
\end{corollary} 

\begin{proof} If ${}_S^+R={}_S^tR$, then $\mathrm{Spec}(S)\to\mathrm{Spec}(R)$ is bijective because so are $\mathrm{Spec}(S)\to\mathrm{Spec}({}_S^tR)$ and $\mathrm{Spec}({}_S^+R)\to\mathrm{Spec}(R)$ by Proposition \ref{1.31}.

The converse is obvious because ${}_S^+R\neq{}_S^tR$ implies that $\mathrm{Spec}({}_S^tR)\to\mathrm {Spec}({}_S^+R)$ is not bijective according to Theorem \ref{minimal} since there is a minimal decomposed extension as a subextension of ${}_S^+R\subset {}_S^tR$.
\end{proof}
 
\section {First properties of the lattice $[R,S]$}

The lattice of  an FCP extension $R\subseteq S$  is  complete and both Noetherian  and Artinian,  with $R$ as the least element and $S$ as the largest element. 

In the context of a lattice $[R,S]$, some  definitions and properties of lattices have the following formulations. 

An element $T$ of $[R,S]$ is an {\it atom} if and only if $R\subset T$ is a minimal extension. We denote by $\mathcal{A}$ the set of atoms of $[R,S]$. 
  
An extension $R\subset S$ is called: 

(a) {\it catenarian}, or graded by some authors, if $R\subset S$ has FCP and all maximal chains between two comparable elements have the same length (the Jordan-H\"older chain condition) \cite{Pic 12}. This concept was studied for field extensions by Dobbs and Shapiro \cite{DS2}.

(b) {\it distributive} if intersection and product are each distributive with respect to the other. Actually, each distributivity implies the other \cite[Exercise 5, page 33]{NO}. 

(c) {\it Boolean} if $[R,S]$ is a distributive lattice such that each $T\in[R,S]$ has a (necessarily unique) complement. We say that $R\subset S$ is a {\it $B_2$-extension} if $R\subset S$ is a Boolean extension of length 2. 
  
(d) a {\it diamond} if $\ell[R,S]=2$ and $|[R,S]|=5 $. Moreover, \cite[Theorem 1, page 59]{G} says that an extension $R\subset S$ is distributive if and only if it   contains neither a pentagon, that is an extension $V\subset W$ of length 3 such that $|[V,W]|=5$, nor a diamond.  

(e) Let $\mathcal{C}:=\{T_i\}_{i\in\mathbb N_n}\subseteq]R,S[,\ n\geq 1$ be a finite chain. Then, $R\subset S$ is called {\it pinched} at $\mathcal{C}$ if $[R,S]=\cup_{i=0}^n[T_i,T_{i+1}]$, where $T_0:=R$ and $T_{n+1}:=S$, which means that any element of $[R,S]$ is comparable to the $T_i$'s. 
 
 \begin{proposition}\label{1.001} \cite[Figure 1]{LSi} An  extension $R\subset S$ is a $B_2$-extension if and only if $\ell[R,S]=2$ and $|[R,S]|=4$.
 \end{proposition}
 
\begin{proposition} \label{1.0} \cite[Proposition 3.2]{Pic 12}, \cite[Lemma 2.10]{Pic 10} A distributive lattice of finite length is catenarian and has FIP.  
 \end{proposition} 
 Proposition \ref{1.0} states that a distributive FCP extension is catenarian. The following Corollary gives a converse. 
  
\begin{corollary} \label{9.02} A catenarian extension $R\subset S$ is distributive if and only if, for any subextension $V\subset W$ of length 2, then $|[V,W]|\leq 4$, or, equivalently, for any subextension $V\subset W$ of length 2, either  $V\subset W$ is chained, or  $V\subset W$  is  a $B_2$-extension.
\end{corollary}
\begin{proof} One implication is Proposition \ref{1.0}. Conversely, assume that $R\subset S$ is a catenarian extension such that for any subextension $V\subset W$ of length 2, then $|[V,W]|\leq 4$. Since $R\subset S$ is catenarian, it does not contain a pentagon. Assume that there exist $V, W\in[R,S]$ such that $V\subset W$ is a diamond. Then, $|[V,W]|=5$ and $\ell[V,W]=2$, a contradiction with the assumption. It follows that $R\subset S$ is distributive. The last equivalence comes from  Proposition \ref{1.001}. 
\end{proof}

The following Proposition summarize \cite[Propositions 7.1, 7.4 and 7.6] {DPPS}, whereas (1) was proved in \cite[Proposition 2.15]{Pic 13}.
 
\begin{proposition}\label{3.6} Let $R\subset T$ and $R\subset U$ be two distinct minimal integral extensions, whose compositum $S:=TU$ exists. Setting $M:=\mathcal{C}(R,T)$ and $N:=\mathcal{C}(R,U)$, the following statements hold: 
 \begin{enumerate}
 \item If $M\neq N$, then $[R,S]=\{R,T,U,S\}$.
 
\item If $M=N,\ R\subset T$ is inert and $R\subset U$ is not inert, then $R\subset S$ is not catenarian.
 
\item If $M=N$ with $R\subset T$ and $R\subset U$ are  both non-inert, then $R\subset S$ is catenarian and infra-integral. Moreover, we have the following:
 \begin{enumerate}
\item If $PQ\subseteq M$ for some $P\in\mathrm{Max}(T)$ and some $Q\in\mathrm{Max}(U)$ lying above $M$, then $\ell[R,S]=2$.

\item If $PQ\not\subseteq M$ for any $P\in\mathrm{Max}(T)$ and any $Q\in\mathrm{Max}(U)$ lying above $M$, then $\ell[R,S]=3$.  
 \end{enumerate}
  \end{enumerate}
  \end{proposition} 
  
\begin{proposition} \label{1.4} An FCP extension is distributive if and only if  each subextension $T\subseteq U$ of $R\subseteq S$ is distributive.
\end{proposition} 
 \begin{proof} Obvious.
 \end{proof}  
  
Before giving new results about the distributivity of ring extensions, we sum up some results of \cite{Pic 10}. 
  
\begin{proposition}\label{1.014} \cite[Proposition 2.4]{Pic 10} Let $R\subseteq S$ be a ring  extension. The following statements are equivalent: 
\begin{enumerate}
\item $R\subseteq  S$ is distributive;

\item $R_M\subseteq S_M$ is distributive for each $M\in\mathrm{MSupp}(S/R)$;

\item $R_P\subseteq S_P$ is distributive for each $P\in\mathrm{Supp}(S/R)$;

\item $R/I\subseteq S/I$ is distributive for each ideal $I$ shared by $R$ and $S$;

\item $R/I\subseteq S/I$ is distributive for some ideal $I$ shared by $R$ and $S$.
\end{enumerate}
\end{proposition}
 
\begin{proposition} \label{desc} \cite[Proposition 3.7]{Pic 10} Let $R\subset S$  be a ring extension, $f:R\to R'$ a faithfully flat ring morphism and $S':=  R'\otimes_RS$. If $R'\subset S'$ is distributive, then, $R\subset S$ is distributive.
   \end{proposition}
   
\begin{remark} We deduce from the above statement, that the distributive property is local on the spectrum. This means that for a ring extension $R \subset S$ and any finite set $\{r_1,\ldots,r_n\}$ of elements of $R$, such that $R =Rr_1+\cdots+Rr_n$, then $R\subset S$ is a distributive extension if and only if all the extensions $R_{r_i}\subset S_{r_i}$ have the distributive property. 
\end{remark}
   
Given a ring $R$, recall that its {\it Nagata ring} $R(X)$ is the localization $R(X)=T^{-1}R[X]$ of the ring of polynomials $R[X]$ with respect to the multiplicatively closed subset $T$ of all polynomials with contents $R$. In \cite[Theorem 32]{DPP4}, Dobbs and the authors proved that when $R\subset S$ is an extension, whose Nagata extension $R(X)\subset S(X)$  has FIP, the map $\varphi:[R,S]\to[R(X),S(X)]$ defined by $\varphi(T)= T(X)$ is an order-isomorphism. 
 
 \begin{proposition}\label{4.1999} An FCP extension $R\subset S$, whose Nagata extension $R(X)\subset S(X)$ is distributive, has FIP and is distributive.
 \end{proposition} 

\begin{proof} Since $R\subset S$ has FCP, so has $R(X)\subset S(X)$ by \cite[Theorem 3.9]{DPP3}. Then, $R(X)\subset S(X)$ has FIP according to Proposition \ref{1.0}. This implies that the map $\psi:[R,S]\to[R(X),S(X)]$ defined by $\psi(T)=T(X)$ is an order-isomorphism by \cite[Theorem 32]{DPP4}, so that $R\subset S$ has FIP. Moreover, using \cite[Corollary 3.5]{DPP3}, the FCP property of $R\subset S$ implies that $S(X)\simeq R(X)\otimes_RS$. It follows that we can use Proposition \ref{desc}, leading that   $R\subset S$ is distributive  because so is $R(X)\subset S(X)$.
\end{proof}

\begin{proposition} \label{desc 1} \cite[Proposition 3.9]{Pic 10} Let $R\subset S$ be a ring extension, $f:R\to R'$ a flat ring epimorphism and $S':=R'\otimes_RS$. If $R\subset S$ is a distributive extension, then so is   $R'\subset S'$.
  \end{proposition}

\begin{definition}\label{4.2} A ring extension $R\subseteq S$ is called 
{\it arithmetic} if $[R_P, S_P]$ is a chain for each $P\in\mathrm{Spec}(R)$.
\end{definition}

 In \cite{Pic 13}, an extension $R\subseteq S$ is called a $\delta$-{\it extension} if $R[x]+R[y]=R[x+y]$ for any $x,y\in S$ such that $R[x]\neq R[y]$.

\begin{proposition} \label{1.004} A chained FCP extension $R\subset S$  is simple.
\end{proposition} 

\begin{proof}  According to \cite[Proposition 5.18]{Pic 13}, a chained extension is a $\delta$-extension, and then simple by \cite[Proposition 5.17(3)]{Pic 13}. 
\end{proof}  

In \cite{Pic 10}, we proved the following results:

\begin{proposition}\label{1.41} \cite[Proposition 2.8]{Pic 10} Let $R\subset S$ be a distributive extension. Let $T\in[R,S]$ be a product of finitely many atoms. Then, $R\subset T$ is simple. More precisely, if $T=\prod_{i=1}^nR[x_i]$, where the $R\subset R[x_i]$ are minimal distinct extensions, then, $T=R[\sum _{i=1}^nx_i]$.  
\end{proposition} 

\begin{proposition}\label{5.4} \cite[Proposition 5.18]{Pic 4}  An arithmetic extension   is distributive. 
\end{proposition}

 We deduce the following obvious Corollary from Proposition \ref{1.014}. See also \cite[Corollary 3.6]{Pic 10}.

\begin{corollary}\label{5.04} \cite[Proposition 5.18]{Pic 4}  An locally minimal extension   is distributive. 
\end{corollary}

\begin{theorem}\label{5.41} A Pr\"ufer extension is distributive. 
\end{theorem}
\begin{proof} Let $R\subseteq S $ be a Pr\"ufer extension. Then so is $R_M\subseteq S_M$ for each $M\in \mathrm{Max}(R)$ by \cite[Proposition 1.1]{Pic 5}. Let $M\in \mathrm{Max}(R)$. We infer from \cite[Proposition 1.2]{Pic 5} that   there exists $P\in \mathrm{Spec}(R_M)$ such that $S_M=(R_M)_P,\ P=S_MP$ and $R_M/P$ is a valuation domain with quotient field $S_M/P$. It follows that $P$ is a common ideal of  $R_M$ and $S_M$, with $R_M/P\subseteq S_M/P$  chained, and so is $R_M\subseteq S_M$ because of the order isomorphism $[R_M, S_M]\to[R_M/P, S_M/P]$ defined by $T\mapsto T/P$. To conclude, $R\subseteq S$ is arithmetic, and then distributive.
\end{proof}

 Let $R\subseteq S$ be a ring extension.  We are going to look at the transfer of the distributivity property  to some extensions deduced from $R\subseteq S$. 

\begin{proposition}\label{5.12} Let $R \subseteq S$ be an  FCP extension. For an ideal $J$ of $S$, set $I:=J\cap R$. If $R \subseteq S$  is distributive, so is $R/I\subseteq S/J$. 
\end{proposition}
\begin{proof} Since $R/I\cong (R+J)/J$, and because the  distributive   property still holds for any subextension, it also holds for $R/I\subseteq S/J$ by Proposition  \ref{1.014}.
\end{proof}

We may remark that the converse does not hold in general. Indeed, in view of the bijection $[R+J,S]\to [R/I,S/J]$, if $R\subseteq R+J$ does not verify the required property, so does not  $R\subseteq S$. 

\begin{proposition}\label{5.13} Let $R \subseteq S$ be an  FCP extension. Assume that $R=\prod_{i=1}^n R_i$ is a product of rings. For each $i\in\mathbb N_n$, there exists ring extensions $R_i\subseteq S_i$ such that $S\cong \prod_{i=1}^n S_i$. Moreover $R \subseteq S$ is distributive  if and only if so are $R_i\subseteq S_i$ for each $i\in\mathbb N_n$. 
\end{proposition}
\begin{proof} The first part of the statement is \cite[Lemma III.3]{DMPP}. 

We recall the following statement of \cite[just before Section 2]{Pic 9}:

 If $R_1,\ldots,R_n$ are finitely many rings,  the ring $R_1\times \cdots \times R_n$ localized at the prime ideal $P_1\times R_2\times\cdots \times R_n$ is isomorphic to $ (R_1)_{P_1}$ for $P_1 \in \mathrm{Spec}(R_1)$. This rule works for any prime ideal of the product. Since a maximal ideal $M\in\mathrm{Max}(R)$ is of the form $R_1\times \cdots \times M_i\times \cdots \times R_n$ for some $i\in\mathbb N_n$ and $M_i\in\mathrm{Max}(R_i)$, we get that $R_M\cong  (R_i)_{M_i}$ and $S_M\cong  (S_i)_{M_i}$. It follows that $M\in\mathrm{MSupp}(S/R)\Leftrightarrow R_M\neq S_M \Leftrightarrow (R_i)_{M_i}\neq  (S_i)_{M_i}\Leftrightarrow M_i\in\mathrm{MSupp}(S_i/R_i)$.

Since  distributivity holds for $R\subseteq S$ if and only if it holds for $R_M\subseteq S_M$ for each $M\in\mathrm{MSupp}(S/R)$ because of Proposition \ref{1.014}, the previous isomorphisms give the last result.
\end{proof}

Let $R$ be a  ring and $M$ an $R$-module. We consider the ring extension $R\subseteq R(+)M$, where $R(+)M$ is the idealization of $M$ in $R$.  Recall that $R(+)M:=\{(r,m)\mid (r,m)\in R\times M\}$ is a commutative ring whose operations are defined as follows: 

$(r,m)+(s,n)=(r+s,m+n)$ \ \   and  \ \ \ $(r,m)(s,n)=(rs,rn+sm)$

Then  $(1,0)$ is the unit of $R(+)M$, and $R\subseteq R(+)M$ is an injective  ring morphism defining $R(+)M$ as an $R$-module, so that we can identify any $r\in R$ with $(r,0)$. 

 We are now looking at the transfer of distributivity to idealizations.
 
\begin{proposition}\label{5.14} Let $R\subseteq S$ be an FCP extension and $M$ an $S$-module. Then $R \subseteq S$  is distributive if and only if so is $R(+)M\subseteq S(+)M$. 
\end{proposition}
\begin{proof} Since $M$ is an $S$-module, it is also an  $R$-module, and we have the ring extension $R(+)M\subseteq S(+)M$. Set $C:=(R:S)$. Obviously, we get $C(+)M=(R(+)M:S(+)M)$. Since $(R(+)M)/(C(+)M)\cong R/C$ and $(S(+)M)/(C(+)M)\cong S/C$, the following holds: $R(+)M\subseteq S(+)M$ is distributive $\Leftrightarrow (R(+)M)/(C(+)M)\subseteq (S(+)M)/(C(+)M)$ is distributive $\Leftrightarrow R/C\subseteq S/C$ is distributive $\Leftrightarrow R\subseteq S$ is distributive   by  Proposition \ref{1.014}.
 \end{proof}

\section{General properties of FCP distributive extensions}

Since an FCP distributive extension is catenarian by Proposition \ref{1.0}, the results of our paper \cite{Pic 12} are useful when characterizing FCP distributive extensions. To this aim, we give here a  characterization of  distributive lattices of finite length given by {\L}azarz and  Sieme{\' n}czuk adapted to the context of ring extensions. 

\begin{proposition} \label{6.1}  An FCP extension $R\subset S$ is distributive if and only if, for any $T,U\in[R,S]$, the following conditions hold:  
 \begin{enumerate}

\item If $T\cap U\subset U$ and $T\cap U\subset T$ are minimal, then $|[T\cap U,TU]|=4$.

\item If $T\subset TU$ and $ U\subset TU$ are minimal, then $|[T\cap U,TU]|=4$.
\end{enumerate}
 \end{proposition}
 
 \begin{proof}   We recall the result of \cite[Theorem, p.  16]{LS}, given in our context: An FCP extension $R\subset S$ is distributive if and only if, for any  $T,U\in[R,S]$, the following conditions hold:  
 \begin{enumerate}

\item If $T\cap U\subset U$ and $T\cap U\subset T$ are minimal, then $T\cap U\subset TU$ is a $B_2$-extension.

\item If $T\subset TU$ and $ U\subset TU$ are minimal, then $T\cap U\subset TU$ is a $B_2$-extension.
\end{enumerate}
But, Proposition \ref{1.001} says that an extension $R'\subset S'$ is a $B_2$-extension if and only if $\ell[R',S']=2$ and $|[R',S']|=4$. Moreover, in both cases (1) and (2), $|[T\cap U,TU]|=4$ implies that $\ell[T\cap U,T]=2$ since $[T\cap U,TU]=\{T\cap U,T,U,TU\}$ because $T\neq U$. Otherwise, $T=U$ implies in both cases $T\cap U=T=U=TU$, a contradiction.
  \end{proof}

\begin{corollary} \label{9.03} An FCP  extension $R\subset S$ pinched at a finite chain $\mathcal{C}:=\{T_i\}_{i\in\mathbb N_n}\subset]R,S[,\ n\geq 1$ where $T_0:=R$ and $T_{n+1}:=S$, is distributive if and only if $T_i\subset T_{i+1}$ is distributive for each $i\in\{0,\ldots,n\}$.
\end{corollary} 

\begin{proof} We have $[R,S]=\cup_{i=0}^{n}[T_i,T_{i+1}]$. One implication is obvious. Conversely, assume that $T_i\subset T_{i+1}$ is distributive for each $i\in\{0,\ldots,n-1\}$.  
 We use Proposition \ref{6.1}. Let $U,V\in[R,S]$. 
 
Assume first that $U\cap V\subset U,V$ are minimal. If $U\cap V\in[T_i, T_{i+1}[$ for some  $i\in\{0,\ldots,n\}$, then $U,V\in[T_i,T_{i+1}]$, so that $|[U\cap V,UV]|=4$ because $T_i\subset T_{i+1}$ is distributive.  The case $U\cap V=S$ is  impossible. 

Assume now that $U,V\subset UV$ are minimal. If $U V\in]T_i,T_{i+1}]$, then $U,V\in[T_i,T_{i+1}]$, so that $|[U\cap V,UV]|=4$ because $[T_i, T_{i+1}]$ is distributive. The case $U V=R$ is impossible. 

To conclude, Proposition \ref{6.1} shows that $R\subset S$ is distributive.
\end{proof}

 \begin{corollary} \label{6.103} A ring extension $R\subset S$ of length 2  is distributive if and only if $|[ R,S]|\leq 4$.
\end{corollary} 
 
 \begin{proof} Since $\ell[R,S]=2$, it follows that $R\subset S$ is an FCP extension. Then the result is obvious: use Proposition \ref{6.1} when $|[R,S]|\geq 4$ because $|[R,S]|>4$ shows that $R\subset S$ is not distributive.
 
If $|[R,S]|<4$, then $[R,S]$ is a chain, and so is distributive by Proposition \ref{5.4}.
 \end{proof}
 
In the following, Proposition \ref{6.1} will often be used. As extensions of length 2 play a significant role in this Proposition, we recall the characterization of length 2 extensions $R\subset S$ with for each case, the value of $|[R,S]|$ we get in our paper  \cite{Pic 6}.

 More precisely, \cite[Theorem 6.1]{Pic 6} gives a complete characterization of length 2 extensions $R\subset S$ with for each case, the value of $|[R,S]|$. 
 
According to Proposition \ref{6.1}, the existence of subextensions $T\cap U\subset TU$ of $R\subset S$ such that $\ell[T\cap U,TU]=2$ and $|[T\cap U,TU]|\neq 4$ implies that $R\subset S$ is not distributive. 
 
Next result of \cite[Theorem 6.1]{Pic 6} allows us to characterize length 2 extensions $R\subset S$ such that $|[R,S]|=4$, that is, $ B_2$-extensions by Proposition \ref{1.001}.

\begin{theorem} \label{3.19} \cite[Theorem 6.1]{Pic 6} A ring extension $R\subset S$ is a $ B_2$-extension if and only if one of the following conditions holds: 
\begin{enumerate}

\item $|\mathrm{Supp}(S/R)|=2,\ \mathrm{Supp}(S/R)\subseteq \mathrm{Max}(R)$. 

\item $R\subset S$ is an infra-integral $M$-crucial extension such that ${}_S^+R\neq R,S$ and  $(R:S)= M$. 

\item $R\subset S$ is a t-closed integral $M$-crucial extension, so that $M=(R:S)$, and the field extension $R/M\subset S/M$ satisfies one of the following conditions:

\noindent  (a) $R/M\subset S/M$ is neither radicial nor separable, nor exceptional.

\noindent (b) $R/M\subset S/M$ is a finite separable field extension such that $|[R/M, S/M]| = 4$.
\end{enumerate} 
\end{theorem}

In \cite{Pic 11}, we studied the Loewy series $\{S_i\}_{i=0}^n$  associated to an FCP ring extension $R\subseteq S$ defined as follows (\cite[Definition 3.1]{Pic 11}): the {\it socle} of the extension $R\subset S$ is $\mathcal S[R,S]:=\prod_{A\in\mathcal A}A$ and the {\it Loewy series} of the extension $R\subset S$ is the chain $\{S_i\}_{i=0}^n$ defined by induction: $S_0:=R,\ S_1:=\mathcal S[R,S]$ and for each $i\geq 0$ such that $S_i\neq S$, we set $S_{i+1}:=\mathcal S[S_i,S]$. Of course, since $R\subset S$ has FCP, there is some integer $n$ such that $S_n=S_{n+1}=S$. 

We introduced a property often involved in  \cite{Pic 11}: An FCP extension $R\subset S$ with Loewy series $\{S_i\}_{i=0}^n$ is said to satisfy the property $(\mathcal P)$ (or is a $\mathcal P$-extension) if $R\subset S$ is  pinched at $\{S_i\}_{i=0}^n$.
   Moreover, we have the following result:

\begin{theorem} \label{8.5} \cite[Theorem 3.22]{Pic 11} An FCP $\mathcal P$-extension  $R\subset S$, with Loewy series $\{S_i\}_{i=0}^n$ is distributive if and only if $S_i\subset S_{i+1}$ is Boolean for each $0\leq i\leq n-1$. If these conditions hold, then $R\subset S$ has FIP.
\end{theorem}

The following Lemma will be useful in the next proofs. The co-integral closure $\underline R$ of an FCP extension $R\subset S$ is involved but does not appear in the statements. It is the least $T\in[R,S]$ such that $T\subseteq S$ is integral. In fact, in \cite[Theorem 3.7]{Pic 14}, we prove that $\underline R=\widetilde R$ when $R\subset S$ is an FCP almost-Pr\"ufer  extension.

\begin{lemma} \label{6.2} Let $R\subset S$ be an FCP almost-Pr\"ufer  extension. The map $\varphi:[\widetilde R,S]\to[R,\overline R]$ defined by $\varphi(T):=T\cap\overline R$ for any $T\in[\widetilde R,S]$ is an order isomorphism and  $\mathrm{Supp}_R(S/\widetilde R)=\mathrm{Supp}_R(\overline R/R)\subseteq\mathrm{Max}( R)$.
\end{lemma}

 \begin{proof} 
The first part of the statement is \cite[Proposition 3.8]{Pic 14}.
 
Let $M\in\mathrm{Supp}_R(S/\widetilde R)$, so that $M\not\in\mathrm{Supp}_R(\widetilde R/R)$ by \cite[Proposition 4.16]{Pic 5}. It follows that $(\widetilde R)_M=\widetilde {R_M}=R_M$ by \cite[Corollary 4.20]{Pic 5}, which leads to $R_M\subset S_M$ is an integral extension because $R_M\subset S_M$ is also almost-Pr\"ufer, and $M\in\mathrm{Supp}_R(\overline R/R)$ because $\overline R_M=S_M$. Then, 
 $\mathrm{Supp}_R(S/\widetilde R)\subseteq\mathrm{Supp}_R(\overline R/R)
 \subseteq\mathrm{Max}(R)$ because $R\subseteq\overline R$ is integral. It follows that $\mathrm{Supp}_R(S/\widetilde R)=\mathrm{MSupp}_R(S/\widetilde R)=\mathrm{MSupp}_R(\overline R/ R)=\mathrm{Supp}_R(\overline R/ R)$ by \cite[Theorem 6.13]{Pic 15}.
 \end{proof}
 
 We are going to characterize  distributive extensions by using the two conditions of  Proposition  \ref{6.1}. Since the proof is quite long, we split it in two propositions. 

\begin{proposition} \label{6.3} Let $R\subset S$ be an FCP extension such that $R\subseteq\overline R$ is distributive. For any $T,U\in[R,S]$ such that $T\cap U\subset U$ and $T\cap U\subset T$ are minimal, then $|[T\cap U,TU]|=4$.
 \end{proposition}

 \begin{proof}  Of course, $T\neq U$.
 Set  $V:=T\cap U,\ M:=\mathcal C(V,T)$ and $ N:= \mathcal C(V,U)$, so that $M,N\in\mathrm{Max}( V)$. The proof  differs  when the ideals $M$ and $N$ are different  or not.  
 
 If $M\neq N$, Proposition \ref{3.6} shows that $|[T\cap U,TU]|=4$.

 If $M= N$, \cite[Lemma 1.5]{Pic 6} shows that either $V\subset U$ and $V\subset T$ are both minimal Pr\"ufer, or $V\subset U$ and $V\subset T$ are both minimal integral. But, $V\subset U$ and $V\subset T$ cannot be both  minimal Pr\"ufer by \cite[Theorem 6.10]{DPP2}, since they have the same crucial maximal ideal. So, the only possibility is that $V\subset U$ and $V\subset T$ are both minimal integral. In particular, $V\subset TU$ is integral. If $V\in[R, \overline R]$, then $T,U\in[R, \overline R]$ and $|[T\cap U,TU]|=4$ according to Proposition  \ref{6.1} because $R\subseteq \overline R$ is distributive.
 Assume that $V\not\in[R, \overline R]$, and let $W$ be the integral closure of $R\subset V$, so that $ W\subset TU$ is almost-Pr\"ufer and $V$ is the Pr\"ufer hull of this extension. Let $W'$ be the integral closure of $W\subset TU$. We have the following commutative diagram:
 $$\begin{matrix}
  {}  & {}   & {} &       {}       & W' & {}             & {}  \\
  {}  & {}   & {} & \nearrow & {}   & \searrow & {}   \\
 R   & \to & W &      {}       & {}   & {}            & TU \\
 {}   & {}   & {} & \searrow & {}   & \nearrow & {}   \\
 {}   & {}   & {} &      {}       & V  & {}             & {}
\end{matrix}$$
Then, there is an order isomorphism $\varphi:[V,TU]\to[W,W']$ defined by $\varphi(X):=X\cap W'$ for any $X\in[V,TU]$, by Lemma \ref{6.2}. Set $T':=\varphi(T)$ and $U':=\varphi (U)$. Since $W=\varphi(V)$, we get that $W\subset T'$ and $W\subset U'$ are minimal, so that $W=T'\cap U'$ and $T'U'=\varphi(TU)=W'$. But $[W,W']\subseteq [R,\overline R]$ because $W\subseteq W'$ is integral as $R\subseteq W$. Hence,  $W\subseteq W'$ is distributive. It follows that $|[ W,W']|=4$ because of Proposition  \ref{6.1}, whence  $|[T\cap U,TU]|=4$.  
 \end{proof}
 
 \begin{proposition} \label{6.4}  Let $R\subset S$ be an FCP extension such that $R\subseteq \overline R$ is distributive. For any distinct $T,U\in[R,S]$ such that  $T\subset TU$ and $ U\subset TU$ are minimal, then $|[T\cap U,TU]|=4$.
 \end{proposition}

 \begin{proof} 
  Of course, $T\neq U$.
 Set $V:=T\cap U,\ M:=\mathcal C(T,TU)$ and $N:=\mathcal C(U,TU)$. We are going to consider the different possibilities for the extensions $T\subset TU$ and $U\subset TU$. Let $V'$ be the integral closure of $V$ in $TU$.
 
(1)  $T\subset TU$ and $ U\subset TU$ are Pr\"ufer minimal. 

We claim that $V\subset TU$ is Pr\"ufer:  $V'\subset TU$ is Pr\"ufer implies that  $T,U\in[V',TU]$ and we get that $V\subseteq V'\subseteq T\cap U=V$, giving $V=V'$. So, $V\subset TU$ is Pr\"ufer, and then distributive by  Theorem  \ref{5.41}, which leads to $|[T\cap U,TU]|=4$ by Proposition  \ref{6.1}.
 
(2)  $T\subset TU$ is Pr\"ufer minimal and $ U\subset TU$ is minimal integral. 

As in (1), we have $V'\subseteq T$. Moreover, $V\subset U$ is not integral, because in this case, $V\subset TU$ would be integral, a contradiction with $T\subset TU$ is Pr\"ufer. Let $V''$ be the integral closure of $V$ in $U$, so that $V''\subseteq V'\subseteq T$. But, we also have $V''\subseteq U$. To conclude, $V\subseteq V''\subseteq T\cap U=V$ leads to $V''=V$. In particular, $V\subseteq U$ is Pr\"ufer and $V\subset TU$ is almost-Pr\"ufer, with $U$ as Pr\"ufer hull. An application of Lemma \ref{6.2} yields that $N\cap V\in\mathrm{Supp}_V(V'/V)=\mathrm{Supp}_V(TU/U)\subseteq\mathrm{Max}(V)$. Using \cite[Proposition 4.26]{Pic 5}, we show that $V\subset T$ is minimal integral. Indeed, $V\subset T$ is minimal integral if and only if $V\cap U=T\cap U\ (*)$ and $VU\subset TU$ is minimal integral $(**)$. But $(*)$ is nothing but $V=V$ while $(**)$ is  equivalent to $U\subset TU$ is minimal integral, which both hold. Then, $V\subset T$ is minimal integral, 
 so that $T$ is the integral closure of $V\subset TU$. Because $V\subset TU$ is almost-Pr\"ufer and according to \cite[Corollary 6.2]{Pic 15}, we have an order-isomorphism $\psi:[V,TU]\to[V,T]\times[T,TU]$ defined by $W\mapsto(W\cap T,WT)$. To conclude, we get $|[T\cap U,TU]|=4$.  
 
 (3) $T\subset TU$ is minimal integral and $ U\subset TU$ is Pr\"ufer minimal. The proof is the same as in (2) exchanging $T$ and $U$. 
 
 (4) $T\subset TU$ and $ U\subset TU$ are both minimal integral.
 
 As in (2), let $V''$ be the integral closure of $V$ in $U$. Then, $V''\subseteq U$ is Pr\"ufer. As  $U\subset TU$ is integral, it follows that $V''\subset TU$ is almost Pr\"ufer, with $U$ as Pr\"ufer hull.  By Lemma \ref{6.2}, we get that $N\cap V''\in\mathrm{Supp}_{V''}(TU/U)\subseteq \mathrm{Max}( V'')$, so that  $N\cap V\in \mathrm{Max}( V)$ since $V\subseteq V''$ is integral. The same reasoning gives that $M\cap V\in \mathrm{Max}( V)$. A new discussion compares $M\cap V$ and $N\cap V$.
 
 (a) $M\cap V\neq N\cap V$.
 
According to \cite[Proposition 5.4]{DPPS}, we get that $V\subset T$ is minimal with crucial maximal ideal $N\cap V$ and $V\subset U$ is minimal with crucial maximal ideal $M\cap V$. Then, we infer from \cite[Lemma 2.7]{DPP2} that  $|[T\cap U,TU]|=4$.

 (b) $M\cap V= N\cap V$.
 
To begin with, we are going to prove that $V\subset TU$ is integral and then,  use a   reasoning similar to the one used  in Proposition \ref{6.3}.
 
 (i) $M$ and $N$ are incomparable.
 
 \cite[Proposition 6.6(a)]{DPPS} tells us that $V\subset U$ is a minimal extension of the same type as $T\subset TU$, and then an integral extension, and so is $V\subset TU$.
  
(ii) $M\subset N$.

From \cite[Proposition 6.6(b)]{DPPS} we infer that $V\subset T$ is a minimal extension of the same type as $U\subset TU$, and then an integral extension, and so is $V\subset TU$.

 (iii) $M= N$.

 \cite[Proposition 5.7]{DPPS} implies  that $M\in\mathrm{Max}( V)$. In particular, $M=(V:TU)$ and $V/M$ is a field, as $T/M$, so that $V/M\subset T/M$ is an FCP field extension, and then an integral extension by \cite[Theorem 4.2]{DPP2}, as $V\subset T $, and then as $V\subset TU$.  
 
 (iv) $N\subset M$: see (ii).

 Then, in each case, $V\subset TU$ is integral. We have the following commutative diagram, where $W$ is the integral closure of $R$ in $V$, and $W'$ is the integral closure of $W$ in $TU$, so that $W\subseteq V$ is Pr\"ufer and $V\subset TU$ is integral. Hence $W\subset TU$ is almost-Pr\"ufer. 
 
\centerline{ $\begin{matrix}
  {}  & {}   & {} &       {}       & W' & {}             & {}  \\
  {}  & {}   & {} & \nearrow & {}   & \searrow & {}   \\
 R   & \to & W &      {}       & {}   & {}            & TU \\
 {}   & {}   & {} & \searrow & {}   & \nearrow & {}   \\
 {}   & {}   & {} &      {}       & V  & {}             & {}
\end{matrix}$}
 
Then, there is an order isomorphism $\varphi:[V,TU]\to[W,W']$ defined by $\varphi(X):=X\cap W'$ for $X\in[V,TU]$, by Lemma \ref{6.2}. Set $T':=\varphi(T)$ and $U':=\varphi(U)$. Since $W'=\varphi(TU)$, we get that $T'\subset W'$ and $U'\subset W'$ are minimal, so that $W'=T' U'$ and $T'\cap U'=\varphi(T\cap U)=\varphi(V)=W$. But $[W,W']\subseteq[R,\overline R]$ because $W\subseteq W'$ is integral as $R\subseteq W$. Hence, $W\subseteq W'$ is distributive. It follows that $|[ W,W']|=4$ thanks to Proposition \ref{6.1}, giving $|[T\cap U,TU]|=4$.  
\end{proof}

\begin{theorem} \label{6.5}  An FCP extension $R\subset S$ is distributive if and only if $R\subset \overline R$  is distributive.
\end{theorem}

\begin{proof} One implication is obvious. For the converse, use Propositions \ref{6.1}, \ref{6.3} and \ref{6.4}.
\end{proof}

\begin{remark} \label{9.71} An extension $R\subset S$ is said {\it almost unbranched} if each $T\in[R,\overline R[$ is a local ring. When  $R\subset S$ is an almost unbranched FCP extension, the result of Theorem  \ref{6.5}  immediately follows.  Indeed, \cite[Corollary 2.6]{Pic 13} shows that $R\subset S$ is pinched at $\overline R$. 
 We may also remark that $\overline R\subseteq S$ is Pr\"ufer since FCP, and then distributive by Theorem \ref{5.41}. It follows that Proposition \ref{9.03} gives the result of Theorem  \ref{6.5}. 

More generally, if $R\subset S$ is a quasi-Pr\"ufer extension pinched at $\overline R$, then $R\subset S$ is distributive if and only if $R\subseteq \overline R$ is distributive. Indeed, $[R,S]=[R,\overline R]\cup[\overline R,S]$, with $\overline R\subseteq S$ Pr\"ufer, and then distributive by Theorem  \ref{5.41}.
\end{remark} 

According to Theorem \ref{6.5}, we can limit our investigation to integral FCP distributive extensions. This is the subject of the following sections. Let $R\subset S$ be an integral FCP extension. According to Proposition \ref{1.014}, $R\subset S$ is distributive if and only if $R_M\subset S_M$ is distributive for any $M\in\mathrm{MSupp}(S/R)$. 
Most of time, we only give a characterization when $R$ is a local ring, because of the complexity of the conditions. 

\section{Distributive integral  IFCP extensions}

We recall that an integral FCP extension is IFCP if all minimal subextensions are minimal of the same type. Let $R\subset S$ be an integral FCP extension where $R$ is a local ring. A characterization of distributivity is given along the paths of the canonical decomposition of $R\subset S$. In this section, we study the distributivity for these following paths, that is for IFCP integral FCP extensions: subintegral extensions (Proposition \ref{6.7}), seminormal infra-integral extensions (Proposition \ref{decomp}) and t-closed extensions (Proposition \ref{tclos}). The reader is warned that deeper results will be given in the next sections, where the unbranched and branched contexts are separated.

To begin with, the following lemma is needed.

\begin{lemma} \label{6.51} Let $R\subset S$ be an infra-integral distributive FCP extension over the local ring $(R,M)$. Then two different minimal subextensions $R\subset T$ and $R\subset U$ of $R\subset S$ cannot have the same type.
\end{lemma}

\begin{proof} Observe that $T\cap U=R$ because $R\subseteq T\cap U\subset T,U$. Since $(R,M)$ is a local ring and $R\subset S$ is integral, then $R\subset S$ is $M$-crucial because $\mathrm{Supp}(S/R)\subseteq\mathrm {Max}(R)$ by \cite[Theorem 3.6]{DPP2}. Using Proposition \ref{6.1}, we get that $|[R,TU]|=4$, so that $\ell[R,TU]=2$. As $R\subset T$ (resp. $R\subset U$) is either ramified or decomposed, Theorem \ref{3.19} shows that the only possible case is that ${}_{TU}^+R\neq R,TU$, so that $R\subset T$ and $R\subset U$ are minimal extensions of different types.
\end{proof}

\begin{proposition} \label{6.7} A subintegral FCP extension $R\subset S$ is distributive if and only if it is arithmetic.
\end{proposition}

\begin{proof} One implication is Proposition \ref{5.4}. Conversely, suppose that $R\subset S$ is distributive. Assuming that $(R,M)$ is a local ring by Proposition \ref{1.014}, it is enough to show that $[R,S]$ is a chain. Let $R=
R_0\subset\cdots\subset R_i\subset\cdots\subset R_n=S$ be a maximal chain, where $R_{i-1}\subset R_i$ is minimal for each $i\in\mathbb{N}_n$. We claim that $[R,S]=\{R_i\}_{ i=0}^n$. Otherwise, there exists $T\in[R,S]\setminus\{R_i\}_{i=0}^n$. In particular, there exists a greatest $k<n$ such that $R_k\subset T$, and then some $T_k\in[R_k,T]\setminus\{R_i\}_{i=k}^n$ such that $R_k\subset T_k$ is minimal ramified, as $R_k\subset R_{k+1}$, a contradiction by Lemma \ref{6.51}. To conclude, $[R,S]$ is a chain.
\end{proof}

We get as a corollary a characterization of distributive modules, by using our results gotten for ring extensions. An $R$-module $M$ is said {\it distributive} if the lattice $\llbracket M \rrbracket$ of its submodules is distributive. In \cite[Definition 1, p.119]{KZ}, a submodule $N$ of an $R$-module $M$ is called {\it distributive} if $N\cap(N'+N'')=(N\cap N')+(N\cap N'')$ for any two submodules $N',N''$ of $M$. Of course, $M$ is distributive if and only if any submodule $N$ is distributive, which is equivalent, by \cite[Proposition 5.2, page 119]{KZ} to $\llbracket M_P\rrbracket$ is a chain, for any $P\in\mathrm{Max}(R)$. We recover this result by a completely different way for modules of finite length (in order to deal with FCP extensions).
   
\begin{corollary}\label{6.24} An $R$-module $M$ of finite length is distributive if and only if $\llbracket M_P\rrbracket$ is a chain, for any $P\in\mathrm{Max}(R)$. 
\end{corollary}
  
\begin{proof} 
Since $M$ is of finite length, \cite[Proposition 2.2]{Pic 8} shows that $R\subseteq R(+)M$ has FCP. According to \cite[Remark 2.9]{D1}, there is a bijection $\varphi:\llbracket M\rrbracket\to[R,R(+)M]$ defined by $\varphi(N)=R (+)N$ for $N\in\llbracket M\rrbracket$. Obviously, for any $N,N'\in\llbracket M\rrbracket$, we have $\varphi(N+N')=\varphi(N)\varphi(N')$ and $\varphi(N\cap N')=\varphi(N)\cap\varphi(N')$. Then, $\varphi$ is a lattice isomorphism. In particular, $M$ is a distributive module if and only if $R\subseteq R(+)M$  is a distributive ring extension. But, in \cite[Lemma 2.1]{Pic 8}, we proved that $R\subseteq R(+)M$ is a subintegral extension. Then, $R\subseteq R(+)M$ is distributive if and only if $R\subseteq R(+)M$ is arithmetic according to Proposition \ref{6.7}. Let $P\in\mathrm{Max}(R)$. Then $P(+)M$ is the only maximal ideal of $R(+)M$ lying over $P$. It follows that $(R(+)M)_P=(R(+)M)_ {(P(+)M)}=R_P(+)M_P$, the last equality coming from \cite[Corollary 25.5]{Hu}. Hence, $R\subseteq R(+)M$ is arithmetic if and only if $[R_P,R_P(+)M_P]$ is a chain, which is equivalent to $\llbracket M_P \rrbracket$ is a chain. To conclude, we get that $M$ is distributive if and only if $\llbracket M_P\rrbracket$ is a chain, for any $P\in\mathrm{Max}(R)$. 
\end{proof}

\begin{proposition} \label{decomp} A seminormal infra-integral FCP extension $R \subset S$ is distributive if and only if it is locally minimal.
\end{proposition}

\begin{proof} One implication is Proposition \ref{5.4}. Conversely, assume that $R\subset S$ is distributive. We can assume that $(R,M)$ is a local ring by Proposition \ref{1.014}. Then, it is enough to show that $R\subset S$ is minimal. Let $T\in[R,S]$ be such that $R\subset T$ is minimal decomposed. According to Lemma \ref{6.51} and Proposition \ref{1.31}, there does not exist any $U\in [R,S]\setminus\{T\}$ such that $R\subset U$ is minimal. If $T\neq S$, there exists some $V\in]T,S]$ such that $T\subset V$ is minimal decomposed, so that $|[R,V]|=3$ and $\ell[R,V]=2$ by \cite[Lemma 5.4]{DPP2}, with $R\subset V$ infra-integral seminormal according to Proposition \ref{1.31}. But \cite[Theorem 6.1 (7)]{Pic 6} shows that $|[R,V]|=5$, a contradiction. To conclude, $T=S$ and $R\subset S$ is minimal.
\end{proof}

As we did in earlier papers for some other properties of t-closed FCP extensions, the case of distributive t-closed FCP extensions can be reduced to the study of distributive finite field extensions. This is shown in the next result. Distributive finite field extensions are examined in  Section 8. 
 
\begin{proposition}\label{tclos} A t-closed FCP extension is distributive if and only if  its residual extensions  are distributive.
\end{proposition}
\begin{proof}
Let $R\subseteq S$ be a t-closed FCP extension. By Proposition \ref{1.014}, $R\subseteq S$ is distributive if and only if $R_M\subseteq S_M$ is distributive for each $M\in\mathrm{MSupp}(S/R)$. Hence we can  assume that $(R,M)$ is a local ring. Then, $M=(R:S)$ and $(S,M)$ is local by \cite[Lemma 3.17]{DPP3}. Because of the lattice isomorphism $[R,S]\to [R/M,S/M]$ defined by $T\mapsto T/M$, we get that $R\subseteq S$ is distributive if and only if $R/M\subseteq S/M$ is distributive. 
\end{proof} 

\section{ Locally unbranched integral  distributive FCP extensions}

We recall that a ring extension $R\subset S$ is called {\it unbranched} if $\overline R$ is local; so that, if in addition the extension is integral, $R$ is a local ring. When $R\subset S$ is quasi-Pr\"ufer and unbranched, all rings in $[R,S]$ are local by \cite[Lemma 3.29]{Pic 11}. A {\it locally unbranched} extension is an extension $R\subset S$ such that $R_P\subset S_P$ is unbranched for each $P\in\mathrm{Spec}(R)$. A locally unbranched integral FCP extension is nothing but an integral $i$-extension (equivalently, a u-closed extension (Definition \ref{1.3})). 

\begin{theorem} \label{6.11} 
A locally unbranched integral FCP extension $R\subset S$ is distributive if and only if the following conditions are satisfied:
 \begin{enumerate}
\item $R\subseteq{}_S^tR$ is  arithmetic;

\item $[ {}_S^tR,S]$ is distributive;

\item  $R\subset S$ is pinched at ${}_S^tR$.
\end{enumerate}
\end{theorem} 

\begin{proof} To begin with, we assume that $R$ is a local ring, so that $R \subset S$ is unbranched. Then ${}_S^+R={}_S^tR$ by Corollary \ref{1.311}.

Assume that $R\subset S$ is distributive. Then (1) is a consequence of Proposition \ref{6.7} and (2) is obvious. Moreover, in this case, $R\subset S$ is catenarian. Then, \cite[Theorem 4.13]{Pic 12} gives (3).

Conversely, assume that (1), (2) and (3) hold. Then, $R\subset S$ is distributive by  Propositions \ref{6.7} and \ref{9.03}.

When $R$ is not  local, apply the beginning of the proof  to $R_M\subset S_M$ for any $M\in\mathrm{MSupp}(S/R)$, Definition \ref{1.3} and Proposition \ref{1.014}.
\end{proof}

\begin{proposition} \label{6.6} Let $R\subset S$ be a distributive locally unbranched (equivalently u-closed) FCP-extension. If ${}_{\overline R}^tR\to\overline R$ is locally minimal, then $R\subset S$ is arithmetic. 
 \end{proposition} 

\begin{proof} By definition, $R\subset S$ is arithmetic means that $R_M\subset S_M$ is chained for each $M\in\mathrm{MSupp}(S/R)$. Then, it is enough to assume that $R$ (and then $\overline R$) are local rings,  and either ${}_{\overline R}^tR=\overline R$, or ${}_{\overline R}^tR\subset\overline R$ is minimal. Moreover, ${}_{\overline R}^+R={}_{\overline R}^tR$ according to Corollary \ref{1.311} because $\mathrm{Spec}(\overline R)\to \mathrm{Spec}( R)$ is bijective. 
By Proposition \ref{6.7}, we know that $[R,{}_{\overline R}^tR]$ is a chain. If ${}_{\overline R}^tR=\overline R$, then $[R,\overline R]$ is a chain. Assume that ${}_{\overline R}^tR\neq\overline R$. If ${}_{\overline R}^tR=R$, then $R\subset \overline R$ is minimal and then chained. If ${}_{\overline R}^tR\neq R,\overline R$, assume that there exists some $T\in[R,\overline R]\setminus([R,{}_{\overline R}^tR]\cup \{\overline R\})$. It follows that $R\subset T$ is not subintegral. Set $U:={}_T^tR\in[R,{}_{\overline R}^tR[$. Then $U\neq T$ and there exist $V\in[U,{}_{\overline R}^tR]$ and $W\in[U,T]$ such that $U\subset V$ is minimal ramified and $U\subset W$ is minimal inert, so that there exist two maximal chains from $U$ to $VW$ of different lengths by Proposition \ref{3.6}, which is absurd because a distributive extension is catenarian (see Proposition \ref{1.0}). Then, $[R,\overline R]=[R,{}_{\overline R}^tR]\cup\{\overline R\}$ is a chain. Using \cite[Lemma 1.5]{Pic 6}, a similar argumentation shows that there does not exist some $T'\in[R,S]\setminus([R,\overline R]\cup [\overline R,S])$. Then, $[R,S]$ is a chain since so is $[\overline R,S]$ by \cite[Proposition 3.30]{Pic 11}. 
\end{proof} 

\section{Branched integral FCP distributive extensions}

To get a general characterization of distributive extensions over a local ring, it remains to consider integral FCP extensions over a local ring which are not unbranched. Then, in case $R$ is a local ring and $R\subset S$ is not unbranched, we say that $R\subset S$ is {\it branched}. In particular, since $|\mathrm{Max}(R)|=1$, it follows that $|\mathrm{Max}(S)|\geq 2$. In this case, ${}_S^tR\neq{}_S^+R$ because of Proposition \ref{1.31}.

We recall (Definition \ref{1.3}) that for an infra-integral extension $R\subset S$, the u-closure ${}_S^uR$ of $R\subset S$ is the least $V\in[R,S]$ such that $V\subseteq S$ is subintegral.  

\begin{lemma} \label{6.91} Let $R\subset S$ be an infra-integral branched FCP extension
such that $[R,{}_S^+R]$ is a chain and $|\mathrm{Max}(S)|=2$. Then $R\subset S$ is pinched at ${}_S^uR\cap{}_ S^+R$ with $[R,S]=[R,{}_S^+R]\cup[{}_S^uR,S]$ and $[R,{}_S^+R]\cap[{}_S^u R,S]=\emptyset$.
\end{lemma} 
  
\begin{proof}  Let $T\in]R,S[$, so that $|\mathrm{Max}(T)|\leq 2$. It follows that $|\mathrm{Max}(T)|=2$ if and only if $T\subseteq S$ is subintegral if and only if $T\in[{}_S^uR,S]$. And $|\mathrm{Max}(T)|=1$ if and only if $R\subseteq T$ is subintegral, or equivalently,  $T\in[R,{}_S^+R]$. To conclude, $[R,S]=[R,{}_S^+R]\cup[{}_S^uR,S]$ and $ [R,{}_S^+R]\cap[{}_S^uR,S]=\emptyset$.
 
We claim that $R\subset S$ is pinched at ${}_S^+R\cap{}_S^uR$. Let $T\in[R,S]=[R,{}_S^+R]\cup[{}_S^uR,S]$. If $T\in[{}_S^uR,S]$, then $T\in[{}_S^+R\cap{}_S^uR,S]$. If $T\in[R,{}_S^+R]=[R,{}_S^+R\cap{}_S^uR]\cup[{}_S^+R\cap{}_S^uR,{}_S^+R]$ since $[R,{}_S^+R]$ is a chain, then either $T\in [R,{}_S^+R\cap{}_S^uR]$ or $T\in[{}_S^+R\cap{}_S^uR,{}_S^+R]\subseteq [{}_S^+R\cap{}_S^uR,S]$. In both cases $T\in[R,{}_S^+R\cap{}_S^uR]\cup[{}_S^+R\cap {}_S^uR,S]$.
\end{proof}

\begin{proposition} \label{6.10} An infra-integral branched FCP extension $R\subset S$
 such that ${}_S^+R\neq R$, is distributive if and only if the following conditions hold:
\begin{enumerate}
\item$[R,{}_S^+R]$ is a chain. 
\item$|\mathrm{Max}(S)|=2$.
\item The map $\varphi:[{}_S^uR\cap{}_S^+R,{}_S^+R]\to[{}_S^uR,S]$ defined by $T\mapsto ({}_S^uR)T$ is an order isomorphism.
\end{enumerate} 

If these conditions hold, then either $|\mathrm{MSupp}_{{}_S^uR}(S/{}_S^uR)|=1$ and $[{}_S^uR,S]$ is a chain or $[R,S]$ is a chain with ${}_S^uR=S$.
 \end{proposition} 

\begin{proof} Assume that $R\subset S$ is distributive. Then, $R\subset{}_S^+R$ is chained by Proposition \ref{6.7}, ${}_S^+R$ is local, ${}_S^+R\subset S$ is minimal decomposed by Proposition \ref{decomp} and $|\mathrm{Max}(S)|=2$. Then (1) and (2) hold.
 
We prove (3). Assume that ${}_S^uR\neq S$. First, $\{({}_S^uR)T\mid T\in[{}_S^uR\cap{}_S^+R,{}_S^+R]\}\subseteq[{}_S^uR,S]$. Now, let $V\in[{}_S^uR,S]$ and set $T:=V\cap{}_S^+R$ so that $T\in[{}_S^uR\cap{}_S^+R,{}_S^+R]$. Since $R\subset S$ is distributive, we have $T({}_S^uR)=[V({}_S^uR)]\cap[({}_S^+R)({}_S^uR)]$. But $({}_S^+R)({}_S^uR) =S$ according to \cite[Proposition 6.12]{Pic 14} and $V({}_S^uR)=V$, giving $T ({}_S^uR)=V$, and $[{}_S^uR,S]=\{({}_S^uR)T\mid T\in[{}_S^uR\cap{}_S^+R,{}_S^+R]\}$.
 
Let $\varphi:[{}_S^uR\cap{}_S^+R,{}_S^+R]\to[{}_S^uR,S]$ defined by $T\mapsto({}_S^uR)T$. We have just proved that $\varphi$ is surjective. Let $T,T' \in[{}_S^uR\cap{}_S^+R,{}_S^+R]$ be such that $\varphi(T)=\varphi(T')$, that is $({}_S^uR)T=({}_S^uR)T'$. Since $R\subset S$ is distributive, we get that $[({}_S^uR)T]\cap{}_S^+R=[({}_S^uR)T']\cap{}_S^+R=[{}_S^uR\cap{}_S^+R][T\cap {}_S^+R]=[{}_S^uR\cap{}_S^+R][T'\cap{}_S^+R]=T=T'$ because ${}_S^uR\cap {}_S^+R\subseteq T,T'\subseteq{}_S^+R$. Then, $\varphi$ is injective, and so an order isomorphism.  

 Assume now that ${}_S^uR=S$. Then, $[{}_S^uR\cap{}_S^+R,{}_S^+R]=\{{}_S ^+R\}$ and $[{}_S^uR,S]=\{S\}$, so that $\varphi$ is obviously an order isomorphism.  

Conversely, let us show that $R\subset S$ is distributive when conditions (1), (2) and (3) hold. If ${}_S^uR=S$, Lemma \ref{6.91} gives that $[R,S]$ is a chain. Then $R\subset S$ is distributive by Proposition \ref{5.4} because arithmetic.
 
Assume now that ${}_S^uR\neq S$. According to Lemma \ref{6.91}, $[R,S]=[R,{}_S^+R]\cup[{}_S^uR,S]$ and $[R,{}_S^+R]\cap[{}_S^uR,S]=\emptyset$. We have the following commutative diagram $\mathcal D$ where the horizontal maps are subintegral and the vertical maps are minimal decomposed:
 $$\begin{matrix}
 {} & {}  &          {}_S^uR            & \to & ({}_S^uR)R_i & \to &       S       \\
 {} & {}  &          \uparrow           &  {} &    \uparrow     &  {}  & \uparrow \\
R & \to & {}_S^uR\cap{}_S^+R & \to &         R_i        & \to & {}_S^+R            
\end{matrix}$$
Set $[{}_S^uR\cap{}_S^+R,{}_S^+R]:=\{R_i\}_{i=0}^n$, with $R_i\subset R_{i+1}$ minimal ramified for each $i\in\{0,\ldots,n-1\}$ by (1)  since $R\subseteq{}_S^+R$ is subintegral (Proposition~\ref{1.31}).

 By (3), we have $[{}_S^uR,S]=\{({}_S^uR)R_i\}_{i=0}^n$, which is a chain. Moreover, $R_i\subset({}_S^uR)R_i$ is minimal decomposed for each $i\in\{0,\ldots,n\}$ because $|\mathrm{Max}(R_i)|=1,\ |\mathrm{Max}(({}_S^uR)R_i)|=2$ and there is no $T\in]R_i,({}_S^uR)R_i[$ such that $R_i\subset T$ is minimal ramified. Otherwise, such $T$ would be in $[{}_S^uR\cap{}_S^+R,{}_S^+R]$, a contradiction.  
 
In order to get the distributivity of $R\subset S$, we use Proposition \ref{6.1} and show that its conditions are fulfilled.
 Let $V,W\in[R,S],\ V\neq W$. Assume first that $V\cap W\subset V,W$ are both minimal. Looking at diagram $\mathcal D$, we observe that necessarily $ V\cap W\in[R_0,{}_S^+R[$ and one of $V,W$ is in $]R_0,{}_S^+R]$ while the other is in $[{}_S^uR,S]$. More precisely, $V\cap W=R_j$, for some $n>j\geq 0 $, with, for instance, $V=R_{j+1}$ and $W=({}_S^uR)R_j$. Then, $VW=R_{j+1}({}_S^uR)R_j=({}_S^uR)R_{j+1}$. Obviously, $[V\cap W,VW]=[V\cap W,V,W, VW]$, so that $|[V\cap W,VW]]=4$. Now, let $V,W\in[R,S],\ V\neq W$ be such that $V,W\subset VW$ are minimal. The same reasoning as above shows that we have necessarily $\{V,W\}=\{({}_S^uR)R_j,R_{j+1}\}$ for some $n>j\geq 0$, and the same proof leads to $|[V\cap W,VW]]=4$. To conclude, $R\subset S$ is distributive by Proposition~\ref{6.1}.

Since ${}_S^+R\neq S$, it follows that $|\mathrm{Max}({}_S^uR)|=2$. Then, we get that $0<|\mathrm{MSupp}_{{}_S^uR}(S/{}_S^uR)|\leq 2$ because ${}_S^uR\neq R$. 
 Set $\mathrm{Max}({}_S^uR)=:\{N,N'\}$ and assume that $|\mathrm{MSupp}_ {{}_S^uR}(S/{}_S^uR)|=2$. According to \cite[Lemma 2.15]{Pic 10}, there exist $V,V'\in[{}_S^uR,S]$ such that ${}_S^uR\subset V$ and ${}_S^uR\subset V'$ are minimal with $N:=({}_S^uR:V)$ and $N':=({}_S^uR:V')$, so that $V\neq V'$, a contradiction with the fact that $[{}_S^uR, S]$ is a chain by (1) and (3). If ${}_S^uR=R$, then $R\subset S$ is pinched at ${}_S^+R$ and $R\subset S$ is chained since ${}_S^+R\subset S$ is minimal (decomposed).
  Then, $|\mathrm{MSupp}_{{}_S^uR}(S/{}_S^uR)|= 1$.
 \end{proof}
 
We use \cite[Example 5.16]{Pic 13} to give an example of the situation of Proposition \ref{6.10} in the context of extensions of number field orders.
\begin{example}\label{13} Consider the quartic number field $K$ defined by the irreducible polynomial $X^4+22X+66$. Let $S$ be the ring of integers of $K$. We get in \cite[Example 5.16]{Pic 13} that $3S=P_1P_2^3$, where $P_1$ and $P_2$ are the maximal ideals of $S$ lying above $3\mathbb Z$. Set $R:=\mathbb Z+3S$. 
We also proved that $R\subset S$ is an infra-integral extension with $3S=(R:S)\in\mathrm{Max}(R)$, so that $\mathrm{MSupp}(S/R)=\{3S\}$ and $R\subset S$ is $3S$-crucial. We are going to show that $R\subset S$ is distributive. 

Using results of \cite[Example 5.16]{Pic 13}, we have the following commutative diagram, where ${}_S^+R=R+P_1P_2,\ U:=R+P_2^3,\ R_1:=R+ P_1P_2^2$ and $UR_1=R+P_2^2 $. The vertical maps are minimal decomposed extensions and the horizontal maps are minimal ramified extensions. In particular, $R\subset S$ is neither subintegral nor seminormal, so that $U={}_S^uR$ since $U\subset S$ is subintegral.
  
 \centerline{$\begin{matrix}
     U       & \to &    UR_1   & \to &     S         \\
\uparrow & {}  & \uparrow & {}  & \uparrow  \\
    R        & \to &    R_1     & \to & {}_S^+R   
    \end{matrix}$}
   It was proved in \cite[Example 5.16]{Pic 13} that $\ell[R,S]=3$. In order to show that $R\subset S$ is distributive, we are going to check the three conditions of Proposition \ref{6.10} for the extension $R_{3S}\subset S_{3S}$. We begin to show that $[R,S]=\{R,U, R_1,UR_1,{}_S^+R,S\}$. To this aim, it is enough to find all the $W,W'\in[R,S]$ such that $W\subset S$ and $R\subset W'$ are minimal. It was also proved in \cite[Example 5.16]{Pic 13} that the only $W\in[R,S]$ such that $W\subset S$ is minimal is either $UR_1$ or ${}_S^+R$, with $UR_1\subset S$ minimal ramified and ${}_S^+R\subset S$ minimal decomposed. We also got that $U$ is the only $W'\in[R,S]$ such that $R\subset W'$ is minimal decomposed. It remains to find $W''\in[R,S]$ such that $ R\subset W''$ is minimal ramified. We proved in \cite[Example 5.16]{Pic 13} that such a $W''$ satisfies $\ell[R,UW'']=2$, so that $UW''\in\{UR_1,{}_S^+R\}$. Moreover, $U\subset UW''$ is minimal ramified, as $UW''\subset S$, which leads to $UW''=UR_1$. We also have that $W''\subset UW''=UR_1$ is minimal decomposed, so that $W''={}_{UR_1}^+R=R_1$.
  
Then, $[R,S]=\{R,U,R_1,UR_1,{}_S^+R,S\}$. Localizing the extension at $M:= 3S=(R:S)\in\mathrm{Max}(R)$ and setting $R':=R_M,\ S':=S_M,\ U':=U_M=({}_ S^uR)_M={}_{S'}^uR',\ R_1':=(R_1)_M,\ U_1':=(UR_1)_M$ and $T':=({}_S^+R) _M={}_{S'}^+R'$, Proposition \ref{6.10} shows that $R'\subset S'$ is distributive. Indeed, $[R',T']$ is a chain, $|\mathrm{Max}(S')]=2$ and, since $[{}_{S'}^uR'\cap{}_{S'}^+R',{}_{S'}^+R']$ and $[{}_{S'}^uR',S']$ are both chains of length 2, (3) of Proposition \ref{6.10} holds. Then, $R\subset S$ is distributive by Proposition \ref{1.014} since $M$-crucial.
     \end{example} 

We are able to exhibit the Loewy series of a distributive infra-integral FCP 
extension $R\subset S$ in case $R\subset S$ is not chained.

\begin{corollary} \label{6.102} If $R\subset S$ is a distributive infra-integral branched FCP extension that is not chained, then $[R,S]=[R,{}_S^+R]\cup[{}_S^uR,S],\ [R,{}_S^+R]\cap[{}_S^uR,S]=\emptyset$ and $[{}_S^uR,S]=\{({}_S^uR)T\mid T\in[R,{}_S^+R]\}$. Moreover, $[R,{}_S^+R]=\{R_i\}_{i=0}^n$ is a chain and there exists some $k\in\{0,\ldots,n- 1\}$ such that $R_k:={}_S^uR\cap{}_S^+R$, so that $[{}_S^uR,S]=\{({}_S^uR) R_i\}_{i=k}^n$ is a chain. The Loewy series of $R\subset S$ is the chain $\{S_i\}_{i=0}^{n+1}$ defined by $S_i=R_i$ for $i\in\{0,\ldots,k\}$ and $S_i=({}_S^uR)R_i$ for $i\in\{k+1,\ldots,n\}$.
      \end{corollary}

\begin{proof} Since $R\subset S$ is an FCP distributive infra-integral extension, $[R,{}_S^+R]=\{R_i\}_{i=0}^n$ is a chain.   In particular, ${}_S^+R\neq S$ because $R\subset S$ is branched.  Moreover, ${}_S^+R\neq R$. Otherwise, $R\subset S$ would be minimal decomposed by Proposition \ref{decomp}.
Since, $R\subset S$ is not chained and ${}_S^+R\neq R,S$, Proposition \ref{6.10} and Lemma \ref{6.91} say that ${}_S^uR\subset S$ is a subintegral chained extension with $[R, S]=[R,{}_S^+R]\cup[{}_S^uR,S],\ [R,{}_S^+R]\cap[{}_S^uR,S]=\emptyset$ and $[{}_S^u R, S]=\{({}_S^uR)T\mid T\in[R,{}_S^+R]\}$. Set $R_k:=U\cap{}_S^+R$ and let $\{S_i\}$ be the Loewy series of $R\subset S$. Since $[R,R_k]$ is a chain, we get that $S_i=R_i$ for $i\in\{0,\ldots,k\}$. Now, $[R_k,S]$ has two atoms ${}_S^ uR$ and $R_{k+1}$ so that $S_{k+1}=\mathcal S[S_k,S]=({}_S^uR)R_{k+1}$. At last, $S_i=({}_S^uR)R_i$ for $i\in\{k+1,\ldots,n\}$ since $[{}_S^uR,S]=\{({}_S^uR)R_i\}_{i=k}^n$  is a chain.
\end{proof}

Now, we look at the seminormal part.

\begin{proposition} \label{6.8} Let $R\subset S$ be a seminormal branched FCP extension. If $R\subset S$ is integral, then $R\subset S$ is distributive if and only if ${}_S^tR\subset S$ is distributive and $[R,S]=\{R\}\cup [{}_S^tR,S]$.
\end{proposition}

\begin{proof} Assume first that $R\subset S$ is distributive. Then  obviously, $[{}_S^tR,S]$ is distributive. Moreover, $R\neq {}_S^tR$ since $R\subset S$ is branched. As $R\subset{}_S^tR$ is distributive and also seminormal infra-integral, then minimal decomposed by Proposition \ref{decomp}. We claim that $R\subset S$ is pinched at ${}_S^tR$. Otherwise, there exists $T\in[R,S]\setminus([R,{}_S^tR]\cup[{}_S^tR,S])$. Since $R\subset T$ is not infra-integral, set $U:={}_T^tR\neq T$, because $T\not\in[R,{}_S^tR]$. Moreover, $U\neq{}_S^tR$ because $T\not\in[{}_S^tR,S]$. It follows that $U=R$ and there exists $T'\in[R,T]$ such that $R\subset T'$ is minimal inert. This implies that $R\subset{}_S^tRT'$ is not catenarian by Proposition \ref{3.6}, a contradiction with Proposition \ref{1.0}. Since $R\subset{}_S^tR$ is minimal, we get that $[R,S]=\{R\}\cup [{}_S^tR,S]$.

Conversely, assume that $[{}_S^tR,S]$ is distributive and $[R,S]=\{R\}\cup[{}_S ^tR,S]$. Then, $[R,S]=[R,{}_S^tR]\cup[{}_S^tR,S]$, with $R\subset{}_S^tR$ distributive because minimal and $[{}_S^tR,S]$ is distributive. Then $[R,S]$ is distributive by Corollary \ref{9.03}.
\end{proof}

Let $R\subset S$ be an FCP extension with $\mathrm{MSupp}(S/R)=\{ M_i\}_ {i=1}^n$. In \cite{Pic 10}, we say that $R\subset S$ is a $\mathcal B$-extension if the map $[R,S]\to\prod_{i=1}^n[R_{M_i},S_{M_i}]$ defined by $T\mapsto(T_ {M_i})_{i=1}^n$ is bijective. An integral FCP extension is a $\mathcal B$-extension \cite[Theorem 3.6]{DPP2}.
 
Let $U,V\in[R,S]$. Then, $V$ is called a {\it complement} of $U$ (in the sense of lattices) if $R=U\cap V$ and $S=UV$. If this complement is unique we denote it by $U^o$.
 
If $R\subset S$ is an extension and $T\in]R,S[$, we say that the extension {\it is split} at $T$, if $\mathrm{MSupp}(S/T)\cap\mathrm{MSupp}(T/R)=\emptyset$  
\cite[definition 4.2]{Pic 15}.

\begin{definition} \label{split2}Let $R\subseteq S$ be a ring extension and $X\subseteq\mathrm{MSupp}_R(S/R)$. We say that an element $T$ of $[R,S]$ is a {\em splitter} of the extension at $X$ if $X=\mathrm{MSupp}_R(T/R)$ and $\mathrm{MSupp}_R(S/T)=\mathrm{MSupp}_R(S/R)\setminus X$. Such a splitter $T$ splits the extension at $T$ and each element that splits an extension is a  splitter. In fact, a splitter at $X$ is unique (\cite[Lemma 4.1]{Pic 15}) and we denote it by $\sigma(X)$.
\end{definition}  
 
 The following Proposition recalls some results on splitters from \cite{Pic 15}.

\begin{proposition}\label{1.14} \cite[Theorems 4.6 and 4.8]{Pic 15} Let $R\subset S$ be an FCP extension. Then, $R\subset S$ is a $\mathcal B$-extension if and only if for any $X\subseteq\mathrm{MSupp}(S/R)$, the splitter of the extension at $X$ exists. 

Assume that $R\subset S$ is a $\mathcal B$-extension and let $X\subseteq\mathrm{MSupp}(S/R)$. Set $T:=\sigma(X)$. Then, $T$ has a unique complement $T^o$, which is the splitter at $\mathrm{MSupp}(S/R)\setminus X$. Moreover, $\mathrm{MSupp}(T/R)=\mathrm{MSupp}(S/T^o)$ and $\mathrm{MSupp}(S/T)=\mathrm{MSupp}(T^o/R)$.
\end{proposition}

We may remark that Proposition \ref{1.14} is  trivial when  either $X= \mathrm{MSupp}(S/R)$, where $\sigma(\mathrm{MSupp}(S/R))=S$ or $X=\emptyset$ where $\sigma(\emptyset)=R$.

 Here is a first interesting application of splitters.

\begin{proposition}\label{1.104} Let $R\subset S$ be an FCP extension which splits at $T\in[R,S]$ and $X:=\mathrm{MSupp}_R(T/R)$. The following statements hold:
\begin{enumerate}
\item $T$ has a unique complement $T^o$, which is the splitter at 

\noindent $\mathrm{MSupp}(S/R)\setminus X$. 
\item $R\subset S$ is distributive if and only if $R\subset T$ and $R\subset T^o$ are distributive.
\end{enumerate}
\end{proposition}

\begin{proof} (1) is \cite[Theorem 4.8]{Pic 15}. Set $U:=T^o$ and $Y:=\mathrm {MSupp}_R(S/T)=\mathrm{MSupp}_R(S/R)\setminus X$. Then, $X=\mathrm {MSupp}(S/U)$ and $Y=\mathrm{MSupp}(U/R)$. It follows  that $\mathrm {MSupp}_R(S/R)=X\cup Y$ with $X\cap Y=\emptyset$.

(2) Let $M\in\mathrm{MSupp}_R(S/R)$. Then, either $M\in X$ or $M\in Y$, these conditions being mutually exclusive. If $M\in X\ (*)$, then $M\not\in Y$, so that $R_M=U_M$ and $T_M=S_M$. If $M\in Y\ (**)$, then $M\not\in X$, so that $R_M=T_M$ and $U_M=S _M$. According to Proposition \ref{1.014}, we have $R\subset S$ is distributive if and only if $R_M\subset S_M$ is distributive for each $M\in\mathrm{MSupp}_R(S/R)$. Let $M\in\mathrm{MSupp}_R(S/R)$. 

In case $(*),\ R_M\subset S_M$ is distributive if and only if $R_M\subset T_M$ is distributive and $R_M=U_M$, so that $R_M\subseteq U_M$ is distributive.

In case $(**),\ R_M\subset S_M$ is distributive if and only if $R_M\subset S_M$ is distributive and $R_M=T_M$, so that $R_M\subseteq T_M$ is distributive.

To conclude,  for each $M\in\mathrm{MSupp}_R(S/R)$, we have that $R_M\subset S_M$ is distributive if and only if $R_M\subseteq T_M$ and $R_M\subseteq U_M$ are distributive, which is equivalent to $R\subseteq T$ and  $R\subseteq U$ are distributive.
\end{proof}

 In the following, we consider arbitrary branched integral distributive FCP extensions $R\subset S$, so that $R$ is a local ring and $S$ does  not need to be a local ring. 
When $R\subset{}_S^tR$ is distributive, we use notation of Proposition \ref{6.10} applied to $R\subset{}_S^tR$. We start by assuming that $R\subset S$ is not pinched at ${}_S^tR$ and $R\subset{}_S^tR$ is not pinched at ${}_S^+R$. The cases where $R\subset S$ is pinched at ${}_S^tR$ or $R\subset{}_S ^tR $ is pinched at ${}_S^+R$ are considered in Theorem \ref{6.15} and Lemma \ref{6.122}.

\begin{proposition} \label{6.12} Let $R\subset S$ be an integral distributive branched FCP extension. We assume in addition that $R\subset S$ is not pinched at ${}_S^tR$ and $R\subset {}_S^tR$ is not pinched at ${}_S^+R$.  
  Setting $T:={}_S^tR$ and $U:={}_S^uR$,  the following properties hold:
  \begin{enumerate}
\item $|\mathrm{Max}(U)|=|\mathrm{Max}(T)|=|\mathrm{Max}(S)|=2$ and $|\mathrm{Supp}_U(T/U)|=1$. Moreover, $[U,T]=:\{U_i\}_{i=0}^n$ is a maximal chain and $U\subset S$ is not pinched at $T$.  
\item If $\mathrm{Max}(U)=\{M,M'\}$ with $\mathrm{Supp}_U(T/U)=:\{M\}$, let $V\in[U,S]$ be the splitter of $[U,S]$ at $\{M'\}$ and set $W:=VT$. 

Then, $[V,W]=\{V_i\}_{i=0}^n$ is a maximal chain such that $V_i=VU_i$ for each $i\in\{0, \ldots,n\}$. Moreover, $U\subset V$ is t-closed, $V={}_W^uU$, $|\mathrm{Supp}_U(S/U)|=2$ and $[U,S]=[U,W]\cup [T,S]$.
\item Let $X\in[R,S]\setminus([R,T]\cup[T,S])$. There exists a unique $i\in\{0,\ldots,n\}$ such that $X\cap T=U_i$ and $X\in[U_i,V_i]$. 
\item $[R,S]=[R,{}_S^+R]\cup[T,S]\cup(\cup_{i=0}^{n-1}[U_i,V_i])$ defines a partition of $[R,S]$.
\item  For each $i\in\{0,\ldots,n-1\}$, the map $\varphi_i:[U_i,V_i]\to[T,W]$ defined by $\varphi_i(X)=XT$ for any $X\in[U_i,V_i]$ is an order-isomorphism.
  \end{enumerate}
\end{proposition}  
 
\begin{proof} (1) Since $R\subset T$ is infra-integral, we have $U=T_R^+$ by \cite[Corollary 6.7]{Pic 14}. Applying Proposition \ref{6.10} to the extension $R\subset T$ 
 because $R\subset S$ is branched, we get that $[U,T]=:\{U_i\}_{i=0}^n$ is a maximal chain, $|\mathrm{Max}(T)|=2$ and $|\mathrm{Supp}_U(T/U)|=1$. It follows that $|\mathrm{Max}(U)|=|\mathrm {Max}(S)|=2$ because $U\subset T$ is subintegral and $T\subset S$ is t-closed (Proposition \ref{6.10}).

Assume that $U\subset S$ is pinched at $T$, so that $[U,S]=[U,T]\cup[T,S]$. Since $R\subset S$ is not pinched at $T$, there exists $X\in[R,S]\setminus([R,T]\cup[T,S])$, with $|\mathrm{Max}(X)|\leq 2$. If $|\mathrm{Max}(X)|=2$, then $X\subset S$ is u-closed and $X\in [U,S]=[U,T]\cup[T,S]\subseteq [R,T]\cup[T,S]$, a contradiction. If $|\mathrm{Max}(X)|=1$, set $Y:=X\cap T\subset T$, because $X\not\in[T,S]$, so that $|\mathrm{Max}(Y)|=1$. But $Y\subseteq X$ is t-closed and $Y\subset T$ is infra integral, with $Y$ a local ring, which leads to $Y=X\in[R,T]$ because $R\subset S$ is distributive, and then catenarian (Proposition \ref{1.0}). This is also absurd since $X\not\in [R,T]$. To conclude, $U\subset S$ is not pinched at $T$.

(2) Set $\mathrm{Max}(U)=:\{M,M'\}$ with $\mathrm{Supp}_U(T/U)=:\{M\}$. We claim that $\mathrm{Supp}_U(S/U)=\{M,M'\}$. Otherwise, $M'\not\in\mathrm{Supp}_U(S/U)$, so that $U_{M'}=S_{M'}$. Since $U\subset S$ is distributive, so is $U_M\subset S_M$, which is unbranched. Then, $U_M\subset S_M$ is pinched at $T_M$ according to Theorem \ref{6.11}, and so $U\subset S$ is pinched at $T$, a contradiction with (1) and we get $\mathrm{Supp}_U(S/U)=\{M,M'\}$. Thanks to Proposition \ref{1.14}, there is a unique $V\in[U,S]$ which is the splitter of $U\subset S$ at $\{M'\}$, then such that $\mathrm{Supp}_U(V/U)=\{M'\}\ (*)$ and $\mathrm{Supp}_U(S/V)=\{M\} \ (**)$.

Set $W:=VT\in[T,S]$ and $V_i=VU_i$ for each $i\in\{0,\ldots,n\}$. By $(*)$, we deduce that $V_M=U_M$, so that $W_M=T_M$ and $[V_M,W_M]=[U_M,T_M]=\{(U_i)_M\}_{i=0}^n=\{V_M(U_i)_M\}_{i=0}^n$. By $(**)$, we deduce that $V_{M'}=S_{M'}$, so that $W_{M'}=S_{M'}=V_{M'}$ and $[V_{M'},W_{M'}]=\{S_{M'}\}$. Then, $[V,W]=\{VU_i\}_{i=0}^n$. Moreover, the previous localizations show that $U\subset V$ is t-closed and $V\subset W$ is subintegral. 

 Let $X'\in[U,S]\setminus[T,S]$. Then, $Y':=X'\cap T\neq T$, so that $Y'\subset T$ is subintegral, and $Y'\subseteq X'$ is t-closed. Since $Y'\in[U,T]$, it follows that $\mathrm{MSupp}_U(T/Y')=\{M\}$, so that $\mathrm{MSupp}_U(X'/Y')=\{M'\}$. Then, $Y'_{M'}=T_{M'}$ and $Y'_M=X'_M$. Since $V_{M'}=S_{M'}=W_{M'}$, we get $X'_{M'}\in[U_{M'},W_{M'}]$ and $Y'_M=X'_M\in[U_M,T_M]\subseteq [U_M,W_M]$ giving $X'\in[U,W]$. Then, $[U,S]=[U,W]\cup[T,S]$.

(3) Let $X\in[R,S]\setminus([R,T]\cup[T,S])$ and set $Y:=X\cap T\in[R,T[$ because $X\not\in [T,S]$. Then, $Y\subseteq X$ is t-closed, so that $|\mathrm {Max}(Y)|=|\mathrm{Max}(X)|$. We reuse a part of the proof of (1) and (2). If $|\mathrm{Max}(Y)|=2$, then $Y\in[U,T]$ by (1), so that there exists a unique $i\in\{0,\ldots,n\}$ such that $Y=U_i$. If $\mathrm{Max}(Y)|=1$, then $Y\in[R,{}_ S^+R]$. But $Y\subseteq X$ is t-closed and $Y\subset T$ is infra- integral, with $Y$ a local ring, which leads to $Y=X\in[R,T]$, as in (1), a contradiction. Then, $U\subseteq U_i=Y\subseteq X$, for some $i\in\{0,\ldots,n\}$. In particular, $X_ M=Y_M=(U_i)_M\subseteq(V_i)_M$ and $X_{M'}= S_{M'}=V_{M'}=(V_i)_{M'}$ shows that $X\in[U_i,V_i]$. In particular, $U_i=V_i\cap T$ for each $i\in\{0,\ldots,n\}$.

(4) Thanks to (3), we get $[R,S]=[R,T]\cup[T,S]\cup(\cup_{i=0}^{n-1}[U_i,V_i])$, 
 defining a partition $[R,S]=[R,{}_S^+R]\cup[T,S]\cup(\cup_{i=0}^{n-1}[U_i,V_i])$.

 More precisely, we have the following commutative diagram:
  $$\begin{matrix}
 {} & {}  &            V            & \to &      V_i          & \to &       W      & \to & S \\
 {} & {}  &       \uparrow     &  {} & \uparrow       &  {}  & \uparrow &   {} & {} \\
 {} & {}  &           U            & \to &        U_i         & \to &       T       &   {} & {} \\
 {} & {}  &      \uparrow     & {} &       \uparrow    & {}  & \uparrow &   {} & {} \\
R & \to & U\cap{}_S^+R & \to & U_i\cap{}_S^+R & \to & {}_S^+R  &  {} & {}
\end{matrix}$$

(5) Let $i\in\{0,\ldots,n-1\}$. We begin to show that the map $\varphi_i$, which obviously preserves order, defined on $[U_i,V_i]$ by $\varphi_i(X)=XT$ for any $X\in[U_i,V_i]$ has its range in $[T,W]$. Since $X\subseteq V_i\subseteq W=V T$, we have $T\subseteq XT\subseteq V_iT\subseteq WT=W$, so that $\varphi_i(X)\in[T,W]$. Let $X,X'\in[U_i,V_i]$ be such that $\varphi_i(X)=\varphi_i (X')$. Then, $XT=X'T$. But, $R\subset S$ is distributive, which implies that $(X T)\cap V_i=(X\cap V_i)(T\cap V_i)=XU_i=X=(X'T)\cap V_i=X'$. Hence, $\varphi _i$ is injective. Now, let $Y\in[T,W]$. Then, $Y\cap V_i\in[U_i,V_i]$, and $(Y\cap V_i)T=(YT)\cap(V_iT)=Y\cap V_iT=Y\cap W=Y$, because $W=VT\subseteq V_iT\subseteq W$ leads to $V_iT=W$. This shows that $\varphi_i$ is surjective, hence an order isomorphism. 
\end{proof}

\begin{proposition} \label{6.121} Let $R\subset S$ be a branched integral FCP extension
 such that ${}_S^+R,{}_S^tR\neq R,S$ and such that $[R,{}_S^tR]$ is a chain. Then $R\subset S$ is distributive if and only if ${}_S^tR\subset S$ is distributive and $R\subset S$ is pinched at ${}_S^tR$. 
  \end{proposition} 

\begin{proof} 
If $R\subset S$ is distributive, so is ${}_S^tR\subset S$.  Since $R$ is local, so is ${}_S^+R$, so that ${}_S^+R\subset {}_S^tR$ is minimal by Proposition \ref{decomp}, because ${}_S^+R\neq{}_S^tR$ since $S$ is not local. Then, $|\mathrm{Max}({}_S^tR)|=2$ and $[R,{}_S^tR]=[R,{}_S^+R]\cup\{{}_S^tR\}$. Assume that $[R, S]\neq[R,{}_S^tR]\cup[{}_S^tR,S]$ and let $T\in[R,S]\setminus([R,{}_S^tR]\cup[{}_S^tR,S])$. Set $T':={}_T^tR={}_S^tR\cap T\in[R,{}_S^tR[$. It follows that $T'\subset T$ is t-closed and $T'\subset{}_S^tR$ is infra-integral. But $T'$ is local since $T'\in[R,{}_S^+R]$, because $T'\neq{}_S^tR$, contradicting the distributivity of $R\subset S$, according to Proposition \ref{3.6}(2). Then, $[R, S]=[ R, {}_S^tR]\cup[ {}_S^tR,S]$. 

Conversely, assume that ${}_S^tR\subset S$ is distributive and $[R, S]=[ R,{}_S^tR]\cup[{}_S^tR,S]$. Since $R\subset{}_S^tR$ is chained, it is distributive. Then, $R\subset S$ is distributive by Corollary \ref{9.03}. 
\end{proof}

We are now in position to characterize branched integral distributive extensions $R\subset S$. We use the  notation of Propositions \ref{6.10} and \ref{6.12}. 
 
\begin{lemma}\label{6.122} Let $R\subset S$ be an integral distributive branched  FCP extension. If $R\subset{}_S^tR$ is pinched at ${}_S^+R$, then, $R\subset S$ is pinched at $ {}_S^tR$.  
 \end{lemma}
\begin{proof} Since $R\subset S$ is branched, then ${}_S^+R\neq{}_S^tR$. Assume first that $R\neq{}_S^+R\neq{}_S^tR\neq S$. Since $R\subset S$ is distributive, $R\subset{}_S^+R$ is chained and ${}_S^+R\subset{}_S^tR$ is minimal by Proposition \ref{6.10}. It follows that $R\subset{}_S^tR$ is chained since pinched at ${}_S^+R$. Then, Proposition \ref{6.121} shows that $R\subset S$ is pinched at ${}_S^tR$.

If $R={}_S^+R$, then $R\subset S$ is pinched at ${}_S^tR$ by Proposition \ref{6.8}
 because $R\subset S$ is seminormal.

If ${}_S^tR=S$, then $R\subset S$ is obviously pinched at $ {}_S^tR$.
\end{proof}

\begin{theorem} \label{6.15} Let $R\subset S$ be an integral branched FCP extension.   
 Then, $R\subset S$ is distributive if and only if $|\mathrm{Max}(S)|=|\mathrm {Max}({}_S^uR)|=2,\ R\subset{}_S^tR$ and ${}_S^tR\subseteq S$ are distributive and one of the following conditions is satisfied, when setting $\mathrm{Max}({}_S^uR)=:\{M,M'\}$:
  
   \begin{enumerate}
  \item $R\subset S$ is pinched at ${}_S^tR$.
\item $R\subset S$ is not pinched at ${}_S^tR,\ |\mathrm{Supp}_{{}_S^uR}({}_S^tR/{}_S^uR)|=1$ and 

\noindent $|\mathrm{Supp}_{{}_S^uR}(S/{}_S^uR)|=2$.
 
Let $\mathrm{Supp}_{{}_S^uR}({}_S^tR/{}_S^uR)=:\{M\}$ and let $V\in[{}_S^uR,S]$ be the splitter of $[{}_S^uR,S]$ at $\{M'\}$. Set $W:=V{}_ S^tR$. Then $V$ is the co-subintegral closure of ${}_S^uR\subset W$ and $[V, W]=\{V_i\}_{i=0}^n$ is a chain (resp. $[{}_S^uR,{}_S^tR]$ is a chain $\{U_i\}_{i=0}^m$). Moreover, $[R,S]=[R,{}_S^+R]\cup[{}_S^tR,S]\cup(\cup_{i=0}^{n-1}[V_i\cap{}_S^tR,V_i])$.
 
If these conditions hold, $m=n$ and $U_i=V_i\cap{}_S^tR$ for each $i\in\{0,\ldots,n\}$.
 \end{enumerate}
\end{theorem}  

\begin{proof} 
Assume first that $R\subset S$ is distributive, then  we have either (1) or $R\subset S$ is not pinched at ${}_S^tR$. In this second case,  condition (2) is  fulfilled by Proposition \ref{6.12} and Lemma \ref{6.122}. 

Conversely, assume that $|\mathrm{Max}(S)|=|\mathrm{Max}({}_S^uR)|=2,\ R\subset{}_S^t R$ and ${}_S^tR\subseteq S$ are distributive and one of conditions (1) or (2) is satisfied. 

If (1) is satisfied, then $R\subset S$ is distributive by Corollary \ref{9.03}.

Assume now that  (2) is satisfied. 

Set $X:={}_S^uR$ and $Y:={}_S^tR$. Since $R\subseteq Y$ distributive by assumption, $ R\subseteq{}_S^+R$ is distributive. Moreover, $[R,S]=[R,{}_S^+R]\cup[ Y,S]\cup(\cup_{i=0}^{n-1}[V_i\cap Y,V_i])$ shows that any $T\in[R,S]\setminus[R,{}_ S^+R]$ is in $[ Y,S]\cup (\cup_{i=0}^{n-1}[V_i\cap Y,V_i])$. But $[V_i\cap Y,V_i]\subseteq[X,S]$. It follows that $[R, S]=[R,{}_S^+R]\cup[X,S]$. In particular, $[X,S]=[Y,S]\cup(\cup_{i =0}^{n-1}[V_i\cap Y,V_i])\ (*)$.
 
Since $R\subset Y$ is distributive, $|\mathrm{Max}(S)|=
|\mathrm{Max}(X)|=|\mathrm{Max}(Y)|=2$. Set $\mathrm{Max}(X)=:\{M,M'\}$ such that $\mathrm{Supp}_X(Y/X)=\{M\}$.

Let $V$ be the splitter of $[X,S]$ at $\{M'\}$. Then, $\mathrm{MSupp}_X(S/V)=\{M\}$ and $\mathrm{MSupp}_X(V/X)=\{M'\}$. It follows that $V_M=X_M$ and $W_M=Y_M$, so that $X_M\subseteq W_M$ is subintegral because so is $V\subseteq W$. This implies that $(V_i\cap Y)_M=(V_i)_M$ for each $i\in\{0,\ldots, n\}$. Moreover, $V_{M'}=S_{M'}$ gives that $W_{M'}=S_{M'}$. Then, $ X\subseteq V$ is t-closed, because so is $X_{M'}=Y_{M'}\subseteq S_{M'}=V_ {M'}$. We also get that $V_i\cap Y\subseteq V_i$ is t-closed for each $i\in\{0 ,\ldots,n\}$. Then, $[X_M,S_M]=[Y_M,S_M]\cup\{(V_i)_M\}_{i=0}^n=[W_M,S_M]\cup[V_M,W_M]$ by $(*)$ which shows that $X_M\subseteq S_M$ is distributive because so are $X_M=V_M\subseteq W_M=Y_M$ and $ W_M=Y_ M\subseteq S_M$, with $V_M\subseteq S_M$ pinched at $W_M$. As $V_{M'}=S_{M'}$, it follows that $X_{M'}=Y_{M'}$, so that we get $[X_{M'}, S_{M'}]=[Y_{M'},S_{M'}]$ which leads to $X_{M'}\subseteq S_{M'}$ also distributive since so is $Y\subseteq S$. To conclude, $X\subset S$ is distributive. Moreover, $[X,Y]=\{Y\cap V_i\}_{i=0}^n$ is a chain, with $X=Y\cap V$ because $[X_M,Y_M]=[V_M,W_M]=\{(V_i)_M\}_{i=0}^n=\{(Y\cap V_i)_M\}_{i=0}^n$ and $[X_{M'},Y_{M'}]=\{X_{M'}\}=\{Y_{M'}\cap V_{M'}\}=\{Y_{M'}\cap (V_i)_{M'}\}_{i=0}^n$.
  Conversely, if $[X,Y]$ is a chain, $[X_M,Y_M]=[V_M,W_M]$ and $V_{M'}=W_{M'}$ shows that $[V,W]$ is a chain. In particular, $[R,{}_S^+R]=\{R_j\}_{j=0}^m$ is a chain by Proposition \ref{6.7} since $R\subset Y$ is distributive and there exists some $k\in\{0,\ldots,m\}$ such that ${}_S^+R\cap V=R_k$. We have the following commutative diagram:
 
 \centerline{$\begin{matrix}
{} &  {}  &               V                 & \to &         W      & \to & S \\
{} &  {}  &          \uparrow          &  {} &  \uparrow  &  {} & {} \\
{} &  {}  &              X                  & \to &        Y       &  {} &  {} \\
{} &  {}  &         \uparrow           &  {} &   \uparrow &  {} &  {} \\
R & \to & R_k=X\cap {}_S^+R & \to &  {}_S^+R  &  {} &  {}
    \end{matrix}$}
    
We are going to prove that $R\subset S$ is distributive by considering the fact that $R\subseteq{}_S^+R$ and $X\subset S$ are distributive and checking the two conditions of  Proposition~\ref{6.1}: 

{\bf (A)} If $T\cap U\subset U$ and $T\cap U\subset T$ are minimal, then $|[T\cap U,TU]|=4$.

{\bf (B)} If $T\subset TU$ and $ U\subset TU$ are minimal, then $|[T\cap U,TU]|=4$.

So, let $T,U\in[R,S],\ T\neq U$. We make a discussion considering the different positions of $T$ and $U$ in $[R,S]$ in relation to $[R, {}_S^+R]$ and $[X, S]$.

If $T,U$ are both in either $[R,{}_S^+R]$ or $[X,S]$, then {\bf (A)} and {\bf (B)} hold by Proposition~\ref{6.1}. 

Assume now, for instance, that $T\in[R,{}_S^+R]$ and $U\in[X,S]$. Then, there exists some $j\in \{0,\ldots ,m\}$ such that $T=R_j$. 

{\bf (A)} If $T\cap U\subset U$ and $T\cap U\subset T$ are minimal, then $j>0$ and $T\cap U=R_{j-1}\in[R,{}_S^+R]$, so that the only possibility for $U$ is $U\in[X,Y]\subseteq[R,Y]$. Then, $T,U\in[R,Y]$ which is distributive, so that $|[T\cap U,TU]|=4$.

{\bf (B)} If $T\subset TU$ and $ U\subset TU$ are minimal, the only possibility for $TU$ is that $TU\in[X,Y]$ since $U\not\in[R,{}_S^+R]$. Then, $TU=V_i\cap Y$ for some $i\in\{0,\ldots,n\}$, so that $U\in[X,Y]\subseteq[R,Y]$. As for {\bf (A)}, $T,U\in[R,Y]$ which is distributive, so that $|[T\cap U,TU]|=4$.

The proof is similar if $U\in [R, {}_S^+R]$ and $T\in[X,S]$.

To conclude, {\bf (A)} and {\bf (B)} hold in any case, and then $R\subset S$ is distributive.

 Assume all that these conditions hold. Since $[X,Y]=\{U_i\}_{i=0}^ m$ and $[V,W]=\{V_i\}_{i=0}^n$, we get $[X_M,Y_M]=\{(U_i)_M\}_{i=0}^m=[V_ M,W_M]=\{(V_i)_M\}_{i=0}^n=\{(V_i\cap Y)_M\}_{i=0}^n$. But $X_{M'}=Y_{M'}$ and $V_{M'}=W_{M'}$ show that $m=n=|[X,Y]|=|[V,W]|$. Then, $U_i=V_i\cap Y$ for each $i\in\{0,\ldots,n\}$.
\end{proof}

We can sum up the main results of Theorem \ref{6.15} that do not appear in its statement but appeared        in its proof. 
\begin{corollary}\label{6.153} Let $R\subset S$ be an FCP integral distributive branched extension that is not pinched at ${}_S^tR$. The following statements hold:
\begin{enumerate}
\item $|\mathrm{Max}({}_S^tR)|=2$. Set $\mathrm{MSupp}_{{}_S^uR}({}_S^tR/{}_S^uR)=\{M\}$ and let $\mathrm{Max}({}_S^tR)$

\noindent$=\{N,N'\}$ with $N$ lying above $M$ in ${}_S^tR$. 

\item Assume that $|\mathrm{MSupp}_{{}_S^tR}(S/{}_S^tR)|=2$. Let $W'$ be the splitter of ${}_S^tR\subset S$ at $\{N'\}$ and let $V$ be the co-subintegral closure of $R\subset W'$. Then, ${}_S^uR=V\cap {}_S^tR$. 
\end{enumerate} 
\end{corollary}

\begin{proof} (1) comes from Proposition \ref{6.10} since $R\subset {}_S^tR$ is distributive.

(2) Assume that $|\mathrm{MSupp}_{{}_S^tR}(S/{}_S^tR)|=2$, so that $|\mathrm{MSupp}_{{}_S^uR}(S/{}_S^uR)|=2$. Let $W'$ be the splitter of ${}_S^ tR\subset S$ at $\{N'\}$. Then, $\mathrm{MSupp}_{{}_S^tR}(S/W')=\{N\}$ and $\mathrm{MSupp}_{{}_S^tR}(W'/{}_S^tR)=\{N'\}$. Set $M':=N'\cap{}_S^uR$. It follows that $\mathrm{MSupp}_{{}_S^uR}(S/W')=\{M\}$ and $\mathrm{MSupp}_ {{}_S^uR}(W'/{}_S^tR)=\{M'\}$. Then, $W'_M=({}_S^tR)_M$ and $W'_{M'}=S_ {M'}$. Recall that in the proof of Theorem \ref{6.15}, we get that $W_M=({}_S^tR)_M$ and $W_{M'}=S_{M'}$, which leads to $W'=W$.

 Finally, ${}_S^uR=V\cap{}_S^tR$ was proved in Theorem \ref{6.15}.
    \end{proof}
   
\begin{remark} \label{6.154} Let $R\subset S$ be an FCP integral distributive branched extension such that $R\neq{}_S^+R\neq{}_S^ tR\neq S$ and such that ${}_S^tR\subset S$ is chained. Then, $|\mathrm{MSupp}_{{}_S^tR}(S/{}_S^tR)|=|\mathrm {MSupp}_{{}_S^uR}(S/{}_S^tR)|=1$ because of Proposition \ref{1.14}. We keep notation of Theorem \ref{6.15}, with $\mathrm{MSupp}_{{}_S^uR}(S/{}_S^uR):=\{M,M'\}$ and $\mathrm{MSupp}_{{}_S^uR}({}_S^tR/{}_S^uR):=\{M\}$. Let $N$ (resp. $N' $) be the maximal ideal of ${}_S^tR$ lying above $M$ (resp. $M'$). If $\mathrm{MSupp}_{{}_S^tR}(S/{}_S^tR)=\{N\}$, then, $\mathrm{MSupp}_{{}_S ^uR} S/{}_S^uR)=\{M\}$, and Theorem \ref{6.15} shows that $R\subset S$ is pinched at ${}_S^tR$. If $\mathrm{MSupp}_{{}_S^tR}(S/{}_S^tR)=\{N'\}$, then, $\mathrm{MSupp}_{{}_S^uR}(S/{}_S^tR)=\{M'\}$, and an application of Proposition \ref{1.14} shows that the splitter of $[{}_S^u R,S]$ at $\{M'\}$ is the co-subintegral closure of $R\subset S$. Moreover, $R\subset S$ is not pinched at ${}_S^tR$.
\end{remark}

 \begin{proposition} \label{6.152} Let $R\subset S$ be an FCP integral distributive branched extension such that ${}_S^uR\neq{}_S^tR$. Then $R\subset S$ is pinched at ${}_S^tR$ if and only if $|\mathrm{MSupp}_{{}_S^uR}(S/{}_S^uR)|=1$.
\end{proposition}
\begin{proof} Since $R\subset S$ is an FCP integral distributive branched extension over the local ring $R$,   
$|\mathrm{Max}(S)|=|\mathrm{Max}({}_S^tR)|=|\mathrm{Max}({}_S^uR)|=2$. Moreover, ${}_S^uR\neq{}_S^tR$, so that $|\mathrm{MSupp}_{{}_S^uR}({}_S^tR/{}_S^uR)|=1$ according to Proposition \ref{6.10} and $R\subset {}_S^tR$ is not chained.   

Assume first that $R\neq{}_S^+R$ and ${}_S^tR\neq S$. In particular, $R\subset{}_S^tR$ is not pinched at ${}_S^+R$ since $R\subset{}_S^tR$ is not chained. Set $\mathrm {MSupp} _{{}_S^uR}({}_S^tR/{}_S^uR)=:\{M\}$ and let $M'$ be the other maximal ideal of ${}_S^uR$. 
 
If $R\subset S$ is pinched at ${}_S^tR$,  then $|\mathrm{MSupp}_{{}_S^uR}(S/{}_S^uR)|\neq 2$. Otherwise, let $V$ be the splitter of ${}_S^uR\subset S$ at $\{M'\}$. Since $V_M=({}_S^uR)_M,\ V_{M'}=S_{M'}$ and $({}_S^uR)_{M'}=({}_S^tR)_{M'}$, it follows that ${}_S^uR\subset V$ is t-closed, so that $V\not\in[R,{}_S^tR]\cup[{}_S^tR,S]$, a contradiction. Then, $|\mathrm{MSupp}_{{}_S^uR}(S/{}_S^uR)|\leq 1$. As $S\neq{}_S^uR$, it follows that $|\mathrm{MSupp}_{{}_S^uR}(S/{}_S^uR)|= 1$. 

Conversely, assume that $|\mathrm{MSupp}_{{}_S^uR}(S/{}_S^uR)|=1$ and that $R\subset S$ is not pinched at ${}_S^tR$. Proposition \ref{6.12} (2) shows that $|\mathrm{MSupp}_{{}_S^uR}(S/{}_S^uR)|=2$, a contradiction. Then, $R\subset S$ is  pinched at ${}_S^tR$.
 
If $R={}_S^+R$, then $R\subset S$ is seminormal, so that ${}_S^uR={}_S^tR$, a contradiction.
  
If ${}_S^tR= S$, then $R\subset S$ is infra-integral, and Proposition \ref{6.10} says that $|\mathrm{MSupp}_{{}_S^uR}(S/{}_S^uR)|= 1$ with obviously $R\subset S$ is pinched at ${}_S^tR=S$. 
 \end{proof}
 
We illustrate Theorem \ref{6.15} by an example where we choose
  the simplest case where all the assumptions of the Theorem hold.

\begin{example} \label{6.151}  Let $k\subset K$ be a  minimal field extension and $k\subset k[x]$ be a minimal ramified extension, where $x^2=0$ (\cite[Lemme 1.2]{FO}). Set $S:=k[x]\times K$. We have the following commutative diagram where appear all the elements of $[k,S]$ and all the maps are minimal extensions.  The previous reference shows also that $k\subset k\times k$ is minimal decomposed.
$$\begin{matrix}
 V:=k\times K  & \to &    k[x]\times K=S            \\
   \uparrow     & {}  &             \uparrow             \\
 U:=k\times k & \to & k[x]\times k=W:={}_S^tk  \\
   \uparrow    & {}  &           \uparrow                 \\
        k           & \to &  k+(kx\times 0)={}_S^+k             
\end{matrix}$$
According to \cite[Proposition III.4]{DMPP}, $U\subset W$ and $V\subset S$ are minimal ramified, while $U\subset V$ and $W\subset S$ are minimal inert. Since $k\subset W$
 is infra-integral, it follows that $W={}_S^tk$. Then, \cite[Proposition 2.8]{Pic 9} shows that ${}_S^+k={}_W^+k=k+(kx\times 0)$ because $k\subseteq k$ and $k\subset k[x]$ are both subintegral. Indeed, $k x\times k$ and $k[x]\times 0$ are the maximal ideals of $k[x]\times k$, and their intersection is $kx\times 0$. Since $\ell[k,W]=2$, with $k\subset W$  infra-integral, we get that $k\subset k+(kx\times 0)$ is minimal ramified and $ k+(kx\times 0)\subset W$ is minimal decomposed. The assumptions of Theorem \ref{6.15} are satisfied  so that $k\subset S$ is a distributive extension. In particular, $V$ is the subintegral closure of $k\subset S$ and $U$ is the u-closure of $k\subset S$.
\end{example}

In the following Proposition, we consider an i-extension which is locally unbranched. As  the split property is used, we write it in this section.

\begin{proposition} \label{6.14} Let $R\subset S$ be an FCP integral $i$-extension such that ${}_S^tR\neq S,R,\ |\mathrm{MSupp}(S/R)|=2$ and $R\subset S$ is split at  ${}_S^tR$. The following statements hold:   
\begin{enumerate}
\item The co-subintegral closure $S_R^+$ exists and $R\subset S_R^+$ is t-closed.
\item If ${}_S^tR\subset S$ is distributive, so is $R\subset S_R^+$.
\end{enumerate}
\end{proposition} 
\begin{proof} Since $|\mathrm{MSupp}(S/R)|=2$ with $\mathrm{MSupp}_R({}_S^tR/R)\cap\mathrm{MSupp}_R(S/{}_S^tR)$
 
\noindent $=\emptyset$, set $\mathrm{MSupp}(S/R)=\{M,M'\}$ with $\mathrm {MSupp}({}_S^tR/R)=\{M\}$ and $\mathrm{MSupp}(S/{}_S^tR)=\{M'\}$ because ${}_S^tR\neq S,R$. 
 Moreover, $({}_S^tR)_M=S_M\ (*)$ and $({}_S^tR)_{M'}=R_{M'}\ (**)$.
 
(1) Since $R\subset S$ is split at ${}_S^tR$, the co-subintegral closure $S_R^+$ exists according to \cite[Proposition 6.25]{Pic 14} and $S_R^+=({}_S^tR)^o$, which means that 
 $(S_R^+)({}_S^tR)=S$ and $S_R^+\cap{}_S^tR=R$, so that $R\subset S_R^+$ is t-closed. 
  
(2) Using $(*)$ and $(**)$, and localizing at $M$ and $M'$, we get $(S_R^+)_M\cap ({}_S^tR)_M=R_M=(S_R^+)_M$, so that $R_M\subseteq (S_R^+)_M$ is distributive, and $(S_R^+)_{M'}({}_S^tR)_{M'}=S_{M'}=(S_R^+)_{M'}$. 
  
Assume that ${}_S^tR\subset S$ is distributive. Then so is $({}_S^tR)_{M'}\subset S_{M'}$ by Proposition \ref{1.014}, and so is $R_{M'}\subset(S_R^+)_ {M'}$. Using again Proposition \ref{1.014}, we get that $R\subset S_R^+$ is distributive. 
  \end{proof}
  
\section{Extensions whose   fibers have at most 2 elements}

Our interest in extensions described by the title of the section is justified by the following result.

\begin{proposition} \label{7.0} Let $R\subset S$ be an integral distributive FCP extension. Then, for each $M\in\mathrm{Spec}(R)$, the fiber at $M$ of $R\subset S$ has at most two elements.
  \end{proposition}
\begin{proof} Let $M\in\mathrm{Spec}(R)$. If $M\not\in\mathrm{Supp}(S/R)$, then $R_M=S_M$ and the fiber at $M$ of $R\subset S$ has exactly one  element. Since $R\subset S$ is  an integral  FCP extension, we have $\mathrm {Supp}(S/R)=\mathrm{MSupp}(S/R)$. Now, let $M\in\mathrm{MSupp}(S/R)$. Then, the cardinality of the fiber at $M$ of $R\subset S$ is $|\mathrm{Max}(S_ M)|$. If $R_M\subset S_M$ is unbranched, then $|\mathrm{Max}(S_M)|=1$. If $R_M\subset S_M$ is branched, then $|\mathrm{Max}(S_M)|=2$ by Theorem \ref{6.15}.
\end{proof} 
  
We sum up results on fibers gotten in the previous sections for an integral  extension. First, we give a kind of converse of Proposition \ref{7.0} that is satisfied in the case of   integral distributive FCP extension.
 
\begin{proposition} \label{7.01} Let $R\subset S$ be an integral extension such that for each $P\in\mathrm{Spec}(R)$, the fiber at $P$ of $R\subset S$ has at most two elements. Then, there do not exist $T,U,V\in[R,S]$ such that $T\subset U$ and $U\subset V$ are minimal decomposed, with $|\mathrm {MSupp}_T(V/T)|=1$. 
  \end{proposition}
\begin{proof} Let $M\in\mathrm{MSupp}_T(V/T)$ and set $P:=M\cap R$. 

Since $|\mathrm{MSupp}_T(V/T)|=1$, we get that $M=(T:U)$ is the only element of $\mathrm{MSupp}_T(V/T)$ and there exist $M_1,M_2\in \mathrm {Max}(U)$ such that $M=M_1\cap M_2$ with $M_1\neq M_2$. Moreover, $M= T\cap(U:V)$ according to Lemma \ref {1.9}. Since $U\subset V$ is minimal decomposed, we deduce that $(U:V)=M_i$ for some $i\in\{1,2\}$ so that we may assume that $(U:V)=M_1$. In the same way, there exist $M_{11},M_{12} \in\mathrm{Max}(V)$ such that $M_1=M_{11}\cap M_{12}$ with $M_{11}\neq M _{12}$. As $(U:V)=M_1\neq M_2$, we get that $U_{M_2}=V_{M_2}$ and there exists $M_{22}\in\mathrm{Max}(V)$ lying above $M_2$. Of course, $M_{22} \neq M_{11},M_{12}$, and these three distinct maximal ideals of $V$ are lying above $M$, and then also above $P$. They are all in the fiber of $P$ in $V$, and there exists at least three maximal ideals of $S$ in the fiber of $P$, a contradiction with the assumption. 
\end{proof}    

\begin{corollary} \label{7.02} Let $R\subset S$ be an integral FCP extension such that for each $P\in\mathrm{Spec}(R)$, the fiber at $P$ of $R\subset S$ has at most two elements. Set $U:={}_S^uR$ and $V:={}_U^+R$. Then $V\subseteq U$ is locally minimal. For the same reason, ${}_S^+R \subseteq {}_S^tR$  is locally minimal.
  \end{corollary}
\begin{proof} Set $T:={}_S^tR$. According to Definition \ref{1.3} and Proposition \ref{1.301}, $U\subseteq S$ is an i-extension and $U\subseteq T$. Moreover, $V\subseteq U$ is seminormal infra-integral, so that any maximal chain is composed of minimal decomposed extensions. But for each $P\in\mathrm{Spec}(R)$, the fiber at $P$ of $R\subseteq U$ has at most two elements and $R\subseteq U$ is infra-integral. Let $P\in\mathrm{Spec}(R)$. The fiber at $P$ of $R\subseteq V$ has exactly one element $Q\in\mathrm {Spec}(V)$ and the fiber at $Q$ of $V\subseteq U$ has at most two elements since $U\subseteq S$ is an i-extension. Using Proposition \ref{7.01} and localizing $V\subseteq U$ at $P$, we get that $V_Q\subseteq U_Q$  is either an equality or a minimal decomposed extension. To conclude, $V\subseteq U$ is locally minimal. In fact, there cannot be some $V_1\neq V_2\in[V_Q,U_Q]$ such that $V_Q\subset V_i$ are both minimal decomposed for $i=1,2$, by \cite[Proposition 7.6]{DPPS}, because in this case, there would be some $U_i\neq U_Q$ in $[V_Q,U_Q]$ such that $V_i\subset U_i$ would be minimal decomposed, a contradiction with Proposition \ref{7.01}. 

The last result is obvious.
\end{proof}   

Let $I$ be an ideal of a ring $R$. We say that $I$ is {\it locally maximal} if, for each $P\in\mathrm{Spec}(R)$, either $I_P=R_P$ or $I_P=PR_P$, that is $I_P$ is the maximal ideal of $R_P$. 

\begin{proposition} \label{7.06} Let $R\subset S$ be a seminormal infra-integral extension such that $(R:S)$ is locally maximal and, for each $P\in\mathrm{Spec}(R)$, the fiber at $P$ of $R\subset S$ has at most two elements.  Then $R\subset S$ is a distributive 
 and locally minimal decomposed extension.

Moreover, $R\subset S$ has FCP (and then FIP) if and only if $\mathrm {MSupp}(S/R)$ has finitely many elements.
  \end{proposition}
\begin{proof} Assume first that $(R,M)$ is a local ring, so that $(R:S)=M$. Since $R\subset S$ is seminormal, then $M$ is a radical ideal of $S$ \cite[Lemma 4.8]{DPP2}, and necessarily an intersection of two maximal ideals $M_1,M_2$ of $S$. Otherwise, $M\in\mathrm{Max}(S)$ leads to a contradiction, because in this case, $R/M=S/M$ since $R\subset S$ is infra-integral. Then, using again infra-integrality, we get $R/M\cong S/M_i$ for $i=1,2$, which gives $S/M\cong S/(M_1\cap M_2)\cong S/M_1\times S/M_2\cong(R/M) ^2$. Then, $R\subset S$ is a minimal decomposed extension according to Theorem \ref{minimal}, and then distributive by Corollary \ref{5.04}. 

Now, we show that if $R\subset S$ is a seminormal infra-integral extension such that $(R:S)$ is locally maximal and, for each $P\in\mathrm{Spec}(R)$, the fiber at $P$ of $R\subset S$ has at most two elements, these properties are also satisfied for any extension 
 $R_M\subseteq S_M$, where $M\in\mathrm{Max} (R)$. 

So, let $M\in\mathrm{Max}(R)$ and set $I:=(R_M:S_M)$. First, $R_M\subseteq S_M$ is a seminormal infra-integral extension by \cite[Corollary 2.10]{S} and \cite[Proposition 1.16]{Pic 1} with $(R:S)_M=MR_M$ if $R_M\neq S_M$, which is equivalent to $(R:S)\subseteq M$. If $R_M=S_M$, then $R_M\subseteq S_ M$ is trivially distributive and $M\not\in\mathrm{MSupp}(S/R)$. Assume now that $R_M\neq S_M$, so that $M\in\mathrm{MSupp}(S/R)$. Hence, the fiber at $MR_M$ of $R_M\subset S_M$ has at most two elements. But, $(R:S)_M\subseteq(R_M:S_M)=I$ shows, using \cite[Lemma 4.7]{Pic 0}, that either $(R: S)_M=I\in\mathrm{Max}(R_M)$ or $(R_ M:S_M)=I=R_M$. In this last case, $R_M=S_M$, a contradiction. In the first case, it remains to show that for each $P\in\mathrm{Spec}(R_ M)$, the fiber at $P$ of $R_M\subset S_M$ has at most two elements. Let $P\in\mathrm {Spec}(R_M)$. There exists $Q\in\mathrm{Spec}(R),\ Q\subseteq M$ such that $P=QR_M$. Using \cite[Lemme 1, page 35]{Bki A1}, there is a bijection between the prime ideals of $S$ lying over $Q$ and the prime ideals of $S_M$ lying over $P$. According to the assumption, for each $ P\in\mathrm{Spec}(R_M)$, the fiber at $P$ of $R_M\subset S_M$ has at most two elements. Then, the first part of the proof shows that $R_M\subset S_M$ is a distributive minimal extension, so that $R\subset S$ is a distributive locally minimal extension.

Assume that $R\subset S$ has FCP (and then FIP because distributive). Then, $\mathrm {MSupp}(S/R)$ has finitely many elements by Lemma \ref{1.9}. Conversely, assume that $\mathrm{MSupp}(S/R)$ has finitely many elements. Then, $R\subset S$ has FCP according \cite[Proposition 3.7]{DPP2} since each $R_M\subset S_M$ is minimal for each $\mathrm {MSupp}(S/R)$. 
\end{proof}    

 \begin{definition} \label{7.03} \cite[Definitions 1.1 and 1.2 and Proposition 1.6]{Pic 0} A ring extension $R\subset S$ is {\it u-elementary} if there exists some $s\in S$ such that $S=R[s]$ with $s^2-s,s^3-s^2\in R$ and $ R\subset S$ is {\it u-integral} if there exists a directed set $\{S_i\}\subseteq[R, S]$ such that $S=\cup S_i$ and $R\subset S_i$ is the composite of finitely many u-elementary extensions for each $i$. Moreover, the u-closure ${}_S^uR$ of $R\subset S$ is the greatest element $T$ of $[R,S]$ such that $R\subseteq T$ is u-integral. 
\end{definition}

\begin{proposition} \label{7.05} Let $R\subset S$ be a {\it u-elementary} extension. For each $P\in\mathrm{Spec}(R)$, the fiber at $P$ of $R\subset S$ has at most two elements.
\end{proposition}

\begin{proof}  Let $P\in\mathrm{Spec}(R)$ and $I:=(R:S)$. Then $S/I\cong (R/I)^2$ by \cite[Theorem 2.20]{Pic 0} since $R\subset S$ is a u-elementary extension. According to \cite[Remarks 2.34 (2)]{Pic 0}, we have the following: If $I\not\subseteq P$, then there is only one element of $\mathrm{Spec}(S)$ lying over $P$, and if $I\subseteq P$, there are exactly two elements of $\mathrm{Spec}(S)$ lying over $P$. Then, for each $P\in\mathrm{Spec}(R)$, the fiber at $P$ of $R\subset S$ has at most two elements.
\end{proof}

\begin{remark} \label{7.04} (1) We may remark that a u-elementary extension is a generalization of a minimal decomposed extension. Let $R\subset S$ be a  u-elementary FCP extension. Propositions \ref{7.05} and \ref{1.31} show that there does exist $T,U\in[R,S]$ such that $T\subset U$ is minimal decomposed. But in the following examples, we see that a maximal chain composing a u-elementary FCP extension may contain a minimal extension which is not decomposed, that is ramified. In fact, because of the isomorphism $S/I\cong (R/I)^2$ where  $I:=(R:S)$, $R\subset S$ is infra-integral.

(2) In \cite[Examples 2.32]{Pic 0}, the first author gives the following example. There exists an u-integral extension $R\subset R^2$ where $R$ is a non reduced ring such that $(R:R^2)=0$ is not semi-prime and $R^2$ is an $R$-algebra generated by an idempotent ($(1,0)$ for instance). Then, $R\subset R^2$ is infra-integral not seminormal. If $(R,M)$ is a local ring, the fiber at $M$ of $R\subset R^2$ has exactly two elements: $M\times R$ and $R\times M$.

In \cite[Remark 5.11 (2)]{Pic 14}, we set $R:=(\mathbb Z/2\mathbb Z)[T]/(T^2)= (\mathbb Z/2\mathbb Z)[t]$, where $t$ is the class of $T$ in $R$, and $S:=R\times R[x]$, with $R[x]:=R[X]/(X^2,tX)$ and $x$ is the class of $X$ in $R[X]/(X ^2,tX)$. Then, $R$ is an Artinian local non reduced ring with maximal ideal $M: =Rt$. We prove that ${}_S^uR=R^2$, with $R\subset{}_{R^2}^+R$ and ${}_S^ uR\subset S$ minimal ramified and ${}_{R^2}^+R\subset{}_S^uR$ minimal decomposed. The fiber at $M$ of $R\subset S$ has also exactly two elements.
\end{remark}

\section{Combinatorics for the  integral case}
In this section we give the values of $\ell[R,S]$ and $|[R,S]|$. Thanks to \cite[Theorem 3.6]{DPP2} and \cite[Proposition 4.6]{DPP3}, we have the equations $|[R,S]|=\prod_{M\in\mathrm{MSupp}(S/R)}|[R_M,S_M]|$ and $\ell[R,S]=\sum_{M\in\mathrm {MSupp}(S/R)}\ell[R_M,S_M]$. 
  
  \noindent Therefore, it is enough to get the values when $ R$ is a local ring, in particular because we have the description of $[R,S]$ for a local ring $R$.  

\begin{proposition} \label{7.1} Let $R\subset S$ be an integral distributive  FCP extension over the local ring $R$. Then, 
$$\ell[R,S]=\ell[R,{}_S^+R]+\ell[{}_S^tR,S]+|\mathrm{Max}(S)|-1.$$
  \end{proposition}
\begin{proof} Since $R\subset S$ is a distributive FCP extension, $R\subset S$ is catenarian by Proposition \ref{1.0}. Using the canonical decomposition, we get $\ell[R,S]=\ell[R,{}_S^+R]+\ell[{}_S^+R,{}_S^tR]+\ell[{}_S^tR,S]$. Moreover, ${}_S^+R\subseteq{}_S^t R$ being distributive, we have either ${}_S^+R={}_S^tR\ (*)$ or ${}_S^+R\subset{}_S^tR\ (**)$ is minimal by Proposition \ref{decomp}; that is, either $\ell[{}_S^+R,{}_S^tR]=0$ or $\ell[{}_S^+R,{}_S^tR]=1$. In case $(*),\ |\mathrm{Max}(S)|=1$ by Proposition \ref{1.31}, and in case $(**),\ |\mathrm{Max}(S)|=2$ by  Theorem \ref{minimal}, which gives the result.
\end{proof}    
  
\begin{theorem} \label{7.2} Let $R\subset S$ be an integral distributive FCP extension over the local ring $R$. Then, $|\mathrm{Max}({}_S^uR)|\leq 2$ and $|\mathrm{Supp}_{{}_S^uR}({}_S^tR/{}_S^uR)|\leq 1$.
 
 Consider the following cases:
  \begin{enumerate}
\item $|\mathrm{Supp}_{{}_S^uR}({}_S^tR/{}_S^uR)|=0$.
Then, $|[R,S]|=\ell[R,{}_S^+R]+ |[{}_S^tR,S]|+\lambda$, with $\lambda=0$ when ${}_S^+R=T$ and $\lambda=1$ when ${}_S^+R\neq T$.

\item $|\mathrm{Supp}_{{}_S^uR}({}_S^tR/{}_S^uR)|=1$ and $|\mathrm{Max} ({} _S^uR)|= 2$, in which case we set $\mathrm{Max}({}_S^uR)=:\{M,M'\}$ where $\mathrm{Supp}_{{}_ S^uR}({}_S^tR/{}_S^uR)=:\{M\}$ and denote by $V$ the splitter of ${}_S^uR\subseteq S$ at $\{M'\}$. 
  
 \hskip -1,5cm (a) If $|\mathrm{Supp}_{{}_S^uR}(S/{}_S^uR)|=2$, 
 then, 
 
\centerline{$|[R,S]|=\ell[R,{}_S^+R]+|[{}_S^tR,S]|+\ell[{}_S^uR,{}_S^tR]\cdot|[{}_S^uR,V]|+1\ (*).$}
  
\hskip -1,5cm (b) If $|\mathrm{Supp}_{{}_S^uR}(S/{}_S^uR)|=1$,  then, 

\centerline{ $|[R,S]|=\ell[R,{}_S^+R]+|[{}_S^tR,S]|+\ell[{}_S^uR,{}_S^tR]+1.$}
\item $|\mathrm{Supp}_{{}_S^uR}({}_S^tR/{}_S^uR)|=|\mathrm{Max}({}_S^uR)|=1$, then, 

\centerline{$|[R,S]|=\ell[R,{}_S^+R]+|[{}_S^tR,S]|.$}
  \end{enumerate}
  \end{theorem}
 
\begin{proof} 
Set $T:={}_S^tR$ and $U:={}_S^uR$. Since $R\subset S$ is distributive, so is ${}_S^+R\subseteq T$, which implies that $|\mathrm{Max}(S)|=|\mathrm{Max}(T)|\leq 2$ since $T\subseteq S$ is t-closed, ${}_S^+R$ is a local ring, and either ${}_S^+R\subset T$ is minimal decomposed, according to Proposition \ref{decomp}, or ${}_S^+R=T$. As $U\subseteq S$ is u-closed, then $|\mathrm{Max}(U)|\leq 2$.

We are going to consider all the situations seen in the previous sections, relative to $\mathrm{Max}(U),\ \mathrm{MSupp}_{U}(S/U)$ and $\mathrm{MSupp}_{U}(T/U)$. 
 
 If $R\subset S$ is unbranched, then $|\mathrm{Max}(S)|=1$, so that $|\mathrm{MSupp}_{U}(T/U)|$
 
\noindent $\leq 1$. Assume that $R\subset S$ is branched. If $R\subset S$ is not pinched at $T$, then $|\mathrm{MSupp}_{U}(T/U)|= 1$ according to Theorem \ref{6.15}. If $R\subset S$ is pinched at $T$, then $U\subseteq T$ is subintegral and we still have $|\mathrm{MSupp}_{U}(T/U)|$
 
 \noindent $\leq 1$ by Lemma \ref{6.51}.
 
Assume first that ${}_S^+R=R$ so that $R\subset S$ is seminormal and $[R,S] =\{R\}\cup[T,S]$ by Proposition \ref{6.8}, giving $|[R,S]|= 1+|[T,S]|=\ell[R,{}_S^+R]+|[T,S]|+1$. Moreover, $T=U$, so that $|\mathrm{MSupp}_U(T/U)|=0$. 

Assume now that ${}_S^+R=T$, then $R\subset T$ is subintegral and chained with $R\subset S$ unbranched. It follows that $R\subset S$ is pinched at $T$ by Theorem \ref{6.11}, giving $|[R,S]|=|[R,T]|+|[T,S]|-1=\ell[R,T]+|[T,S]|=\ell[R,{}_S^+R]+|[T,S]|$.
 
 At last, assume  that ${}_S^+R\neq R,T$.

(1) Assume that $|\mathrm{MSupp}_U(T/U)|=0$, so that $U=T$. 
 It follows that $R\subset T$ is pinched at ${}_S^+R$ by Proposition \ref{6.10} and then $R\subset S$ is pinched at $T$ by Lemma \ref{6.122}. Hence, $[R,S] =[R,{}_S^+R]\cup[T,S]$ because ${}_S^+R\subset T$ is minimal decomposed. But, $R\subseteq{}_S^+R$ is chained, which yields $|[R,S]|=\ell[R,{}_S^+R]+|[T,S]|+1$. 

(2) Assume that $|\mathrm{Max}(U)|=2$ and $|\mathrm{MSupp}_U(T/U)|=1$, so that $U\neq T$. Then, $0<|\mathrm{MSupp}_U(S/U)|\leq 2$. Set $\mathrm{Max}(U)=:\{M,M'\}$ and $\mathrm{MSupp}_U(T/U)=:\{M\}$. 
  In particular, $R\subset S$ is branched and $R\subset T$ is not chained.
 
(a) Assume that $|\mathrm{MSupp}_U(S/U)|=2$. Then, ${}_S^+R\neq T$ and $U\neq S$. 

Assume first that $R\subset S$ is not pinched at $T$ and $R\subset T$ is not pinched at ${}_S^+R$. Then, $R\neq{}_S^+R$ and $T\neq S$. We are in the hypothesis of Proposition \ref{6.12}, use its notation and the following diagram $\mathcal D$ of its proof, and recall some of its results: 

  \centerline{$\begin{matrix}
 {} & {}  &            V            & \to &      V_i          & \to &       W      & \to & S \\
 {} & {}  &       \uparrow     &  {} & \uparrow       &  {}  & \uparrow &   {} & {} \\
 {} & {}  &           U            & \to &        U_i         & \to &       T       &   {} & {} \\
 {} & {}  &      \uparrow     & {} &       \uparrow    & {}  & \uparrow &   {} & {} \\
R & \to & U\cap{}_S^+R & \to & U_i\cap{}_S^+R & \to & {}_S^+R  &  {} & {}
\end{matrix}$}
 Let $V$ be the splitter of $U\subset S$ at $\{M'\}$. It follows that
 $[U,T]=:\{U_i\}_{i=0}^n$ is a maximal chain. Set $W:=VT$. Then, $[V,W]=\{V_i\}_{i=0}^n$ is a maximal chain such that $V_i=VU_i$ for each $i\in\{0,\ldots,n\}$. Since $[R,S]=[R,{}_S^+R]\cup[T,S]\cup(\cup_{i=0}^{n-1}[U_i,V_i])$ defines a partition of $[R,S]$, we get that $|[R,S]|=|[R,{}_S^+R]|+|[T,S]|+\sum_{i=0}^{n-1}|[U_i,V_i]|$. Since $R\subset{}_S^+R$ is subintegral distributive, Proposition \ref{6.7} asserts that $R\subset{}_S^+R$ is chained, giving $|[R,{}_S^+R]|=\ell [R,{}_S^+R]+1$. Now for each $i\in\{0,\ldots,n-1\}$, there is an order-isomorphism $[U_i,V_i]\to[T,W]$. In particular, for each $i\in\{0,\ldots,n-1\}$, we have $|[U_i,V_i]|=|[U_0,V_0] |=|[U,V]|$, so that $\sum_{i=0}^{n-1}|[U_i,V_i]|=n|[U,V]|$, with $n=\ell[U,T]$. To sum up, we have $|[R,S]|=\ell[R,{}_S^+R]+1+|[T,S]|+\ell[U,T]\cdot |[U,V]|$, and $(*)$ holds.
  
Assume now that $R\subset T$ is pinched at ${}_S^+R$. Then, Lemma \ref{6.122} shows that $R\subset S$ is pinched at $T$ since $R\subset S$ is branched. But, in this case, Proposition \ref{6.152} shows that $|\mathrm{MSupp}_U(S/U)|=1$, 
  a contradiction since $|\mathrm{MSupp}_U(S/U)|=2$. Then, this situation cannot occur. 

(b) $|\mathrm{MSupp}_U(S/U)|=1$. Since $U\neq T$, Proposition \ref{6.152} shows that $R\subset S$ is pinched at $T$ and $[R,S]=[R,T]\cup[T,S]$, leading to $|[R,S]|=|[R,T]|+|[T,S]|-1$ because $[R,T]\cap[T,S]=\{T\}$. To conclude, $|[R,S]|=\ell[R,{}_S^+R]+|[T,S]|+\ell[U,T]+1$, since $|[R,T]|=\ell[R,{}_S^+R]+\ell[U,T]+2$ by Corollary \ref{6.102} ($[R,T]$ is not chained).

At last $|\mathrm{MSupp}_U(S/U)|\neq 0$,  otherwise  this would  imply $U=S=T$, a contradiction with $|\mathrm{MSupp}_U(T/U)|=1$.

(3) Assume that $|\mathrm{MSupp}_U(T/U)|=|\mathrm{Max}(U)|=1$. Since $U\subseteq S$ is u-closed, it follows that $|\mathrm{Max}(S)|=1$, so that $R\subset S$ is unbranched and ${}_S^+R=T$, a contradiction with the assumption. 

If we consider the beginning of the discussion, we remark that when ${}_S^+R =R$, then $|[R,S]|=\ell[R,{}_S^+R]+|[T,S]|+1$ with $|\mathrm{MSupp}_U(T/U)|=0$, so that the equality of (1) still holds.

When ${}_S^+R=T$, then $|[R,S]|=\ell[R,{}_S^+R]+|[T,S]|$. As $R\subset S$ is unbranched, it follows that $U=R\neq S$. If $T\neq R$, then  $|\mathrm {MSupp}_U(T/U)|=1$ with $|\mathrm{Max}(U)|=1$ and the equality of (3) still holds. If $T=R$, then $R\subset S$ is t-closed and $|\mathrm{MSupp}_U(T/U)|=0$, so that the equality of (1) still holds because ${}_S^+R=T$. 
\end{proof}

  \begin{remark} \label{7.3} We may consider that the equation;

\centerline{$|[R,S]|=\ell[R,{}_S^+R]+|[{}_S^tR,S]|+\ell[{}_S^uR,{}_S^tR]\cdot|[{}_S^uR,V]|+1\ (*)$}
\noindent gotten in Theorem \ref{7.2} (2)(a) also holds if we set $V={}_S^uR$ in cases (1) when ${}_S^+R\neq T$ and (2)(b). When considering these different cases, we see that the diagram $\mathcal D$ is still valid. In case (1) with ${}_S^+R=T$ and in case (3), we have ${}_S^+R=T$, so that $U=R$ and $(*)$  cannot apply here.
\end{remark}
 
\section{Distributive finite field extensions}

According to  Proposition \ref{tclos}, the study of the distributivity of a t-closed FCP extension can be reduced to the study of the distributivity of a finite field extension.

We begin with  the following lemmata. 

\begin{lemma} \label{6.16} Let $k\subset K\subset L$ be a tower of two minimal field extensions, where $k\subset K$ is radicial and $K\subset L$ is separable. Then $k\subseteq L$ is distributive if and only if $|[k,L]|=4$.

If these conditions hold, then $[k,L]=\{k,K,K',L\}$ where $K'$ is the separable closure of $k\subseteq L$.
\end{lemma}

\begin{proof} Clearly, $k\subset L$ is a finite-dimensional field extension. Since $k\subset K$ is minimal radicial, its degree is $p:=\mathrm{c}(k)$. Then, we also have $p=[L:K']$, with $K'\neq k,K,L$, so that $K'\subset L$ is minimal and $|[k,L]|\geq 4$. 

 If $k\subseteq L$ is distributive, it is also catenarian, so that $\ell[k,L]=2$, and \cite[Theorem 6.1 (8)]{Pic 6} shows that $|[k,L]|=4$ because $k\subseteq L$ is neither radicial, nor separable, nor exceptional. Proposition \ref{6.1} gives the converse.
\end{proof}

\begin{lemma} \label{6.17}  Let $k\subset L$ be a finite field extension, $T$  its separable closure and $U$ its radicial closure and assume that $U\subset L$ is separable. Then:

 \begin{enumerate}
\item There exists an order isomorphism $\varphi:[k,T]\to[U,L]$ defined by $\varphi(V):=UV$. 

\item  $k\subset T$ is distributive if and only if $U\subset L$  is distributive.
\end{enumerate}
\end{lemma}
\begin{proof} Define $\varphi:[k,T]\to[U,L]$  by $\varphi(V):=UV$ and  
$\psi:[U,L]\to[k,T]$ by $\psi(V)=V\cap T$. 

(1) Let $V\in[U,L]$ and let $W\in[k,V]$ be the separable closure of $k\subset V$. Then, $W=T\cap V=\psi(V)$. Now, $W\subseteq UW\subseteq V$, with $W\subseteq V$ radicial implies that $UW\subseteq V$ is also radicial. But $U\subseteq UW\subseteq V$, with $U\subseteq V$ separable gives that $UW\subseteq V$ is also separable, so that $V=UW=\varphi(W)=\varphi(\psi(V))$; then, $\varphi\circ\psi$ is the identity on $[U,L]$. 

Now, if $W\in[k,T]$, then  $W\subseteq UW\cap T\subseteq T$ implies that $W\subseteq UW\cap T$ is separable. Since $k\subset U$ is radicial, so is $W\subseteq UW$, and also $W\subseteq UW\cap T$. Then, $W=UW\cap T$, and $\psi\circ\varphi$ is  therefore the identity on $[k,T]$. Since $\varphi$ and $\psi$  preserve order,  $\varphi$ is  an order isomorphism. 

(2) Let $V,V'\in[k,T]$ and  $W,W'\in[U,L]$. Obviously, $\varphi(VV')=\varphi(V)\varphi(V')$, and $\psi(W\cap W')=\psi(W)\cap \psi( W')$. Assume that $W=\varphi(V)$ and $W'=\varphi (V')$, so that $V=\psi(W)$ and $V'=\psi(W')$. Then, $\varphi(V\cap V')=\varphi[\psi(W)\cap \psi( W')]=\varphi[\psi(W\cap  W')]=W\cap W'=\varphi(V)\cap \varphi(V')$. A similar proof shows that $\psi(WW')=\psi(W)\psi(W')$. It follows that  $k\subset T$ is distributive if and only if $U\subset L$  is distributive.
\end{proof}

\begin{proposition} \label{6.18} Let $k\subset L$ be an FIP field extension whose separable closure $T$ is such that $k\subset T$ is distributive. Then $|[U\cap V,UV]|=4$ for any $U,V\in[k,L],\ U\neq V$ such that $U\cap V\subset U,V$ are minimal. 
\end{proposition}
\begin{proof} Since $k\subset T$ is distributive, it is catenarian, and so is $[k,L]$ by \cite[Proposition 4.18]{Pic 12}. 

Set $W:=U\cap V$ and  $W':=W\cap T$. We discuss along   the different types of  minimal extensions $W\subset U$ and $W\subset V$. 

(1) Assume that $W\subset U$ and $W\subset V$ are both radicial.  Then, $W'\subseteq W\subset U,V$ are radicial, and so is $W'\subset UV$. But, by \cite[Lemma 4.1]{Pic 10}, a radicial FIP extension is chained, so that $U=V$, a contradiction.

(2) Assume that $W\subset U$ is radicial and $W\subset V$ is separable, and set $p:=\mathrm{c}(k)$. Then, there is some $x\in U$ such that $U=W[x]$, with $x^p\in W$. It follows that $UV=V[x]$, with $x^p\in V$, so that $V\subset UV$ is minimal radicial. Since $k\subset L$ is catenarian, we get that $\ell[W,UV]=2$. Then, \cite[Theorem 6.1 (8) (c)]{Pic 6} shows that $|[U\cap V,UV]|=4$ because $W\subset UV$ is neither radicial, nor separable, nor exceptional.

(3) Assume that $W\subset U$ and $W\subset V$ are both separable, then so is $W\subset UV$ by \cite[Proposition 8, A.V. 40]{Bki A}. Let  $W''$ be the separable closure of $W'\subseteq UV$. We have the following commutative  diagram
 (only useful   extensions appear):
$$\begin{matrix}
k  & \to & W' &       \to    & W  &      \to      & UV & \to           & S \\
{} & {}  &  {}  & \searrow & {}   & \nearrow &  {}  & \nearrow & {}  \\
{} & {}  &  {}  &       {}      & W'' &      \to      &  T  & {}             & {}
\end{matrix}$$
where $W'\subseteq W$ is radicial as $W''\subseteq UV$, while $W\subset UV$ and $W'\subseteq W''$ are separable, so that $W$ is the radicial closure of $W'\subset UV$. Moreover, $[W',W'']\subseteq [k,T]$ implies that $W'\subseteq W''$ is distributive. Then, an application of Lemma \ref{6.17} to the extension $W'\subset UV$ shows that $W\subset UV$ is distributive, whence, $|[U\cap V,UV]|=4$ by Proposition ~\ref{6.1}.
\end{proof}
\begin{lemma} \label{6.181}  Let $k\subset L$ be an FIP distributive  field extension and $U,V\in[k,L],\ U\neq V$. If $U, V\subset UV$ are minimal separable extensions,  then $U\cap V\subset UV$ is  separable. 
\end{lemma}
\begin{proof} According to Proposition \ref{6.1}, we have $|[U\cap V,UV]|=4$, so that $U\cap V\subset U,V$ are minimal field extension. They are both separable, because  if one of them is radicial, then $|[U\cap V,UV]|>4$, since  $[U\cap V,UV]$ does  contain the separable closure of the extension, distinct from $U\cap V,U,V,UV$. It follows that $U\cap V\subset UV$ is  separable.
\end{proof}

\begin{proposition} \label{6.19} Let $k\subset L$ be an FIP field extension whose separable closure $T$ is such that $k\subset T$ is distributive. Assume that $U\cap V\subset UV$ is separable whenever $U,V\subset UV$ are minimal separable for $U,V\in[k,L],\ U\neq V$. Then $|[U\cap V,UV]|=4$ for any $U,V\in[k,L],\ U\neq V$ such that $U, V\subset UV$ are minimal. 
\end{proposition}

\begin{proof} Set $W:=U\cap V$ and  $W':=W\cap T$, which is the separable closure of $k\subseteq W$. Arguing as in Proposition \ref{6.18}, we discuss along the types of the minimal extensions $U\subset UV$ and $V\subset UV$. 

(1) Assume that $U\subset UV$ and $V\subset UV$ are both radicial.  According to \cite[Proposition 3.8]{DPPS}, $W\subset U$ is radicial, so that $W\subset UV$  is radicial.
  But, by \cite[Lemma 4.1]{Pic 10}, a radicial FIP extension is chained, so that $U=V$, a contradiction.

(2) Assume that $U\subset UV$ is radicial and $V\subset UV$ is separable. Using again \cite[Proposition 3.8]{DPPS}, we see that $W\subset V$ is radicial. Let $T'\in[W,V]$ be such that $T'\subset V$ is minimal radicial. Since $k\subset T$ is distributive, it is catenarian, and so is $k\subset L$ by \cite[Proposition 4.18]{Pic 12}. In particular, $\ell[T',UV]=2$, and an application of \cite[Theorem 6.1 (8) (c)]{Pic 6} shows that $|[T',UV]|=4$. Now, let $W''$ be the separable closure of $T'\subset UV$, so that $W'' \subset UV$ is minimal radicial because $\ell[T',UV]=2$ and $[T',UV]=\{T',V,W'',UV\}$. It follows that $U,W''\subset UV$ are both minimal radicial. Reasoning as in (1), we get that necessarily $W''=U$. Then, $W\subseteq T'\subseteq U\cap V=W$ leads to $T'=W$ and $|[U\cap V,UV]|=4$.

(3) Assume that $U\subset UV$  and $V\subset UV$ are both  separable. It follows from the assumption that $W\subset UV$ is separable. Moreover, $W'\subseteq W$ is radicial because  $W'$ is the separable closure of $k\subseteq W$, so that,  $W\subset UV$ being separable, $W$ is the radicial closure of $W'\subset UV$.  Let $V'$ be the separable closure of $W'\subseteq V$, so that $W'\subseteq V'$ is separable and $V'\subseteq V$ is radicial. But, $V\subset UV$ is separable, then $V$ is the radicial closure of $V'\subset UV$. Let $V''$ be the separable closure of $V'\subseteq UV$, so that $V'\subseteq V''$ is separable, and so is $W'\subseteq V''$, while $V''\subseteq UV$ is radicial. Then, $V''$ is also the separable closure of $W'\subseteq UV$. We have the following commutative diagram (only the maps  involving $V$ are needed):
$$\begin{matrix}
k   & \to & W' & \to         & W & \to          & V &  \to          & UV \\
{} & {}   & {}  & \searrow & {}  & \nearrow & {} & \nearrow & {}  \\
{} & {}  & {}  & {}            & V' & \to           & V'' & {}            & {} 
\end{matrix}$$
Since $W'\subseteq V''$ is separable, we get that $[W',V'']\subseteq[k,T]$. Then, $W' \subseteq V''$ is distributive because so is $k\subset T$. But, an application of Lemma  \ref{6.17} to the extension $W'\subset UV$ shows that $W\subset UV$ is distributive. Finally, Proposition \ref{6.1} shows that $|[U\cap V,UV]|=4$.
\end{proof}

\begin{theorem} \label{6.20} Let $k\subset L$ be an FIP field extension with separable closure $T$. Then, $k\subset L$ is distributive if and only if $k\subset T$ is distributive and $U\cap V\subset UV$ is separable for any $U,V\in[k,L],\ U\neq V$ such that $U,V\subset UV$ are minimal separable extensions.  
\end{theorem}

\begin{proof} One implication is obvious according to Lemma \ref{6.181}. For the converse, use Propositions \ref{6.1}, \ref{6.18} and \ref{6.19}.
\end{proof}

\begin{remark} \label{6.201} The last condition of the previous Theorem holds when $U, V\subset UV$ are minimal separable extensions such that $U\cap V\subset U$ and $U\cap V\subset V$ are linearly disjoint (see \cite[Proposition 10, A.V. 41]{Bki A}).
\end{remark}

To end, we are going to characterize distributive radicial extensions and distributive separable extensions. 

\begin{proposition} \label{6.21}  Let $k\subset L$ be an FCP radicial  field extension. The following conditions are equivalent:
 \begin{enumerate}

\item $k\subset L$ is  chained.

\item $k\subset L$ is distributive.

\item $k\subset L$ has FIP. 
 \end{enumerate}
\end{proposition}

\begin{proof} (1) $\Rightarrow$ (2) by Proposition \ref{5.4}.

(1) $\Rightarrow$ (3) since $k\subset L$ has  FCP.

(3) $\Rightarrow$ (1) by \cite[Lemma 4.1]{Pic 10}.

(2) $\Rightarrow$ (1) Assume that $k\subset L$ is distributive and not chained. There exist $K_1,K_2,K_3\in[k,L],\ K_2\neq K_3$, such that $K_1\subset K_i$ is minimal for $i=1,2$. By Proposition \ref{6.1}, we get that $|[K_1,K_2K_3]|=4 $, with $\ell[K_1,K_2K_3]|=2$, a contradiction with \cite[Theorem 6.1 (8)(a)]{Pic 6} since $K_1\subset K_2K_3$ is radicial. 
\end{proof}

In \cite{Pic 14}, we introduce the following notation. If $\Pi$ is a property of field extensions, we say that a ring extension is a $\kappa$-$\Pi$-extension if all its residual field extensions verify $\Pi$.

\begin{corollary} \label{6.202} Let $R\subset S$ be an FIP $\kappa$-radicial extension. Then $R\subset S$ is $\kappa$-distributive. If, moreover, $R\subset S$ is t-closed, then $R\subset S$ is distributive.
\end{corollary}

\begin{proof} Let $Q\in\mathrm{Spec}(S)$. Then, $\kappa_R(Q\cap R)\to\kappa_S(Q)$ is an FIP radicial extension. According to Proposition \ref{6.22}, we get that $\kappa_R(Q\cap R)\to\kappa_S(Q)$ is a distributive extension, so that $R\subset S$ is $\kappa$-distributive extension. The last result comes from Proposition \ref{tclos}. 
  \end{proof}
  
We recall that a ring extension $R\subseteq S$ is called {\it radicial} if the ring morphism $R\to S$ is universally an $i$-morphism, or equivalently, an $i$-morphism which is $\kappa$-radicial \cite[Definition 3.7.2, p. 248]{EGA} and \cite[Proposition 1]{KU}. 

\begin{corollary} \label{6.203} Let $R\subset S$ be a radicial integral FCP  extension. Then $R\subset S$ is distributive if and only if $R\subseteq{}_S^tR$ is  arithmetic and $R\subset S$ is pinched at ${}_S^tR$.
\end{corollary}

\begin{proof} Since $R\subset S$ is radicial, it is $\kappa$-radicial, and then $\kappa$-distributive by Corollary \ref{6.202}, so that ${}_S^tR\subseteq S$ is distributive according to Proposition \ref{tclos}. Moreover, $R\to S$ is an i-morphism, that is locally unbranched. Then, it is enough to apply Theorem \ref{6.11} to get the result.  
  \end{proof}

For a field $k$, we denote by $k_u[X]$ the set of monic polynomials of $k[X]$. Let $L:=k[x]$ be a finite separable (whence FIP) field extension of $k$ and $ f(X)\in k_u[X]$ be {\it the} minimal polynomial of $x$ over $k$. For any $K\in[k,L]$, we denote by $f_K(X)\in K_u[X]$ the minimal polynomial of $x$ over $K$. We set $\mathcal D:=\{f_K\mid K\in[k,L]\}$. The elements of $\mathcal D$ have been characterized in \cite[Proposition 4.13]{Pic 10}. By \cite[Corollary 4.9]{Pic 10}, there is a reversing order bijection $[k,L]\to \mathcal D$ defined by $K\mapsto f_K$, which gives that $\mathcal D$ is an ordered set for the divisibility and a lattice. In particular, $\sup(f_K,f_{K'})=f_{K\cap K'}$ and $f_{KK'}=\inf(f_K,f_{K'})$ for $K,K'\in[k,L]$.

\begin{proposition}\label{6.22} A finite separable field extension  $k\subset L=k[x]$  is distributive if and only if   $\mathcal D$ is a distributive lattice.
  \end{proposition}
  
\begin{proof} The bijection $\varphi:[k,L]\to\mathcal D$ defined by $\varphi(K)=f_K$ gives the result.
\end{proof}
In case of Galois extension, the result is simpler. 

\begin{proposition}\label{6.23} Let $k\subset L=k[x]$ be a finite Galois extension. Then $k\subset L$ is distributive if and only if   $k\subset L$ is cyclic.
  \end{proposition}
  
\begin{proof} Using the Fundamental Theorem of Galois Theory, there is a reversing order isomorphism between $[k,L]$ and the set $\mathcal G$ of subgroups of the Galois group $G$ of the extension. Then, $k\subset L$ is distributive if and only if $\mathcal G$ is distributive if and only if $G$ is cyclic (see \cite[Example 4.2]{R}) if and only if $k\subset L$ is cyclic.
\end{proof}

As we did for finite Boolean separable extensions in \cite[Theorem 4.19]{Pic 10}, using the normal closure allows to give a result for distributivity.  

\begin{proposition}\label{6.231} Let $k\subset L$ be a finite separable extension with normal closure $N$. If $k\subset N$ is a cyclic extension, then $k\subset L$  is distributive.
\end{proposition}

\begin{proof} In view of Proposition \ref{6.23}, if $k\subset N$ is a cyclic extension, then $k\subset N$  is distributive, and so is  $k\subset L$ by Proposition \ref{1.4}.
\end{proof}

\begin{example}\label{6.232} (1) A cyclic extension with a square free degree is distributive, because Boolean \cite[Theorem 4.19]{Pic 10}. 

(2) Set $k:=\mathbb{Q}$ and $L:=k[x]$, where $x:=\sqrt 3+\sqrt 2$. Then, $k\subset L$ is a Galois extension of degree 4. In \cite[Example 5.17(2)]{Pic 6}, we show that $[k,L]=\{k,k_ 2,k_3,k_6,L\}$, where $k_i:=k[\sqrt i],\ i=2,3,6$ and the following commutative diagram holds with $|[k,L]|=5$:
$$\begin{matrix}
   {}  &        {}      & L             &       {}       & {}     \\
   {}  & \nearrow & \uparrow  & \nwarrow & {}     \\
k_2 &       {}       & k_3         &      {}        & k_6 \\
  {}   & \nwarrow & \uparrow & \nearrow & {}     \\
  {}   &      {}        & k             & {}             & {} 
\end{matrix}$$
Then, $k\subset L$ is not distributive, because it is a diamond \cite[Theorem 1, page 59]{G} or because of Proposition \ref{6.1}.
\end{example}

\section{Applications  of $\Delta$-extensions to distributive  extensions}

We know (Proposition \ref{1.0}) that a distributive extension is catenarian. More precisely, a distributive extension $R\subseteq S$ is obviously {\it modular}; that is, $T_1\cap(T_2T_3)=T_2(T_1\cap T_3)$ for each $T_1,T_2,T_3\in[R,S]$ such that $T_2\subseteq T_1$. In an earlier paper \cite{Pic 13}, we studied $\Delta$-extensions. A ring extension $R\subset S$ is a {\it $\Delta$-extension} if $T+U\in[R,S]$ for any $T,U\in[R,S]$. In particular, a $\Delta$-extension is modular \cite[Proposition 3.14]{Pic 13}   
 and then catenarian. 
We end this paper by showing the link between $\Delta$-extensions and distributive extensions.

\begin{proposition}\label{6.25}   A distributive FCP extension $R\subset S$ is a $\Delta$-extension if and only if $ {}_{\overline R}^tR\subset \overline R$ is arithmetic.
\end{proposition}

\begin{proof} Obvious by \cite[Corollary 4.30]{Pic 13} since a distributive extension is modular.
\end{proof}

\begin{corollary}\label{6.26}  A distributive infra-integral extension  is a $\Delta$-extension.
\end{corollary}

\begin{proof}  Obvious by \cite[Theorem 4.18]{Pic 13} since a distributive extension is modular.
\end{proof}

\begin{proposition}\label{6.27} Let $R\subset S$ be an FCP u-closed   extension. The following conditions are equivalent:
\begin{enumerate}
\item $R\subset S$ is a distributive $\Delta$-extension;

\item $R\subset S$ is an arithmetic $\Delta$-extension;

\item $R\subset S$ is a distributive arithmetic extension;
\end{enumerate}
\end{proposition}

\begin{proof} (1) $\Leftrightarrow$ (3) $\Rightarrow$ (2) is \cite [Proposition 5.4]{Pic 13}. At last, (2) $\Rightarrow$ (1) by Proposition \ref{5.4}.
\end{proof}

\begin{proposition}\label{6.28} Let $R\subset S$ be an FCP integral seminormal $\Delta$-extension, where $R$ is a local ring. Then, $R\subset S$ is distributive if and only if $ [R,S]=\{R\}\cup[{}_S^tR,S]$.
\end{proposition}

\begin{proof} Since $R\subset S$ is an integral $\Delta$-extension,  \cite[Theorem 4.11]{Pic 13} asserts that $R\subset {}_S^tR$ is a $\Delta$-extension and ${}_S^tR\subseteq S$ is arithmetic, so that ${}_S^tR\subseteq S$ is distributive by Proposition \ref{5.4}. Then, $R\subset S$ is distributive if and only if $[R,S]=\{R\}\cup[{}_S^tR,S]$ by Proposition \ref{6.8}.
\end{proof}

We recall that an element $t$ of a ring extension $S$ is said {\it quadratic} over $R$ if $t$ is a zero of a quadratic monic polynomial of $R[X]$.

\begin{proposition} \label{11.12} An extension $R\subset R[t]$, where $t$ is a quadratic element, and such that $R$ is a SPIR, is an FIP distributive extension.
\end{proposition}
\begin{proof} In \cite[Proposition 5.11]{Pic 13}, we proved that $R\subset R[t]$ is an FIP chained extension, then a distributive extension by Proposition \ref{5.4}. 
\end{proof}

\end{document}